\documentclass[10pt]{amsart}

\usepackage{amsmath}
\usepackage{amssymb}
\usepackage{amsthm}
\usepackage{verbatim}
\usepackage[all]{xy}
\usepackage{enumerate}
\usepackage{setspace}

\usepackage{tikz}
\usepackage{graphicx}

\usepackage[all]{xy}

\usepackage{color}

\newtheorem{thm}{Theorem}[section]
\newtheorem{prop}[thm]{Proposition}
\newtheorem{cor}[thm]{Corollary}
\newtheorem{lem}[thm]{Lemma}

\theoremstyle{definition}

\newtheorem{dfn}[thm]{Definition}

\newtheorem{example}[thm]{Example}
\newtheorem{rmk}[thm]{Remark}

\numberwithin{equation}{section}

\newcommand{\cK}{\mathcal{K}}

\newcommand{\cM}{\mathcal{M}}
\newcommand{\cH}{\mathcal{H}}

\newcommand{\CB}{\mathcal{CB}}

\newcommand{\id}{\textrm{id}}

\newcommand{\cB}{\mathcal{B}}

\newcommand{\minotimes}{\otimes_{{\rm min}}}

\newcommand{\cN}{\mathcal{N}}

\newcommand{\boldu}{\mathbf{u}}

\newcommand{\bA}{\mathsf{A}}

\newcommand{\boldg}{{\bf g}}
\newcommand{\boldh}{{\bf h}}

\newcommand{\boldv}{{\bf v}}
\newcommand{\boldw}{{\bf w}}

\newcommand{\boldx}{{\bf x}}

\newcommand{\boldy}{{\bf y}}
\newcommand{\bolda}{{\bf a}}
\newcommand{\boldb}{{\bf b}}
\newcommand{\Star}{\textsf{Star}}

\newcommand{\cQ}{\mathcal{Q}}
\newcommand{\cE}{\mathcal{E}}
\newcommand{\cP}{\mathcal{P}}

\newcommand{\Nor}{\mathsf{Nor}}

\newcommand{\cA}{\mathcal{A}}

\newcommand{\Id}{{\rm Id}}

\newcommand{\Cliq}{\textsf{Cliq}}
\newcommand{\hotimes}{\otimes_h}

\newcommand{\cL}{\mathcal{L}}
\newcommand{\sG}{\mathsf{G}}
\newcommand{\Link}{\textsf{Link}}

\newcommand{\Comm}{\textsf{Comm}}

\newcommand{\oast}{\star}

\title[Absence of Cartan subalgebras for  Hecke von Neumann algebras]{Absence of Cartan subalgebras for right-angled Hecke von Neumann algebras}

\date{\today}

\author{Martijn Caspers}

\begin{document}

\maketitle

\begin{abstract}
For a  right-angled Coxeter system $(W,S)$ and $q>0$, let $\cM_q$ be the associated Hecke von Neumann algebra, which is generated by self-adjoint operators $T_s, s \in S$ satisfying the Hecke relation $(\sqrt{q}\: T_s - q) (\sqrt{q} \: T_s + 1) = 0$ as well as suitable commutation relations. Under the assumption that $(W, S)$ is irreducible and $\vert S \vert \geq 3$ it was proved by Garncarek \cite{Garncarek} that $\cM_q$ is a factor (of type II$_1$) for a range $q \in [\rho, \rho^{-1}]$ and otherwise $\cM_q$ is the direct sum of a II$_1$-factor   and $\mathbb{C}$.

In this paper we prove (under the same natural  conditions as Garncarek) that $\cM_q$ is non-injective, that it has the weak-$\ast$ completely contractive approximation property and that it has the Haagerup property. In the hyperbolic factorial case $\cM_q$ is a strongly solid algebra and  consequently  $\cM_q$ cannot have a Cartan subalgebra. In the general case $\cM_q$ need not be strongly solid. However, we give examples of non-hyperbolic right-angled Coxeter groups such that $\cM_q$  does not possess a Cartan subalgebra.
\end{abstract}

\section{Introduction}
Hecke algebras are one-parameter deformations of group algebras of a Coxeter group. They  were the fundament for the theory of quantum groups \cite{Jimbo}, \cite{Kassel} and have remarkable applications in the theory of knot invariants \cite{JonesKnot} as was shown by V. Jones. A wide range of applications of Coxeter groups and their Hecke deformations  can be found in \cite{Davis}. In \cite{Dymara}  (see also \cite[Section 19]{Davis}) Dymara introduced the von Neumann algebras generated by Hecke algebras. Many important results were then obtained (see also \cite{DymaraEtAl}) for these Hecke von Neumann algebras. This gave for example insight in the cohomology of associated buildings and its Betti numbers. In this paper we investigate the approximation properties of Hecke von Neumann algebras as well as their Cartan subalgebras (here we mean the notion of a Cartan subalgebra in the von Neumann algebraic sense which we recall in Section \ref{Sect=StrongSolidity} and not the Lie algebraic notion).

Let us recall the following definition.  Let $q>0$ and let $W$ be a right-angled Coxeter group with generating set $S$  (see Section \ref{Sect=Prelim}). The associated Hecke algebra is a $\ast$-algebra generated by $T_s, s \in S$ which satisfies the relation:
\[
(\sqrt{q}\: T_s - q) (\sqrt{q} \: T_s + 1) = 0, \qquad T_s^\ast = T_s \qquad  \textrm{ and } \qquad T_s T_t = T_t T_s,
\]
for $s, t\in S$ with  $st = ts$. Hecke algebras carry a canonical faithful tracial vector state (the vacuum state) and therefore generate a von Neumann algebra $\cM_q$ under its GNS construction. It was recently proved by Garncarek \cite{Garncarek}  that if $(W,S)$ is irreducible (see Section \ref{Sect=Prelim}) and $\vert S \vert \geq 3$, the von Neumann algebra $\cM_q$ is a factor in case $q \in [\rho, \rho^{-1}]$ where $\rho$ is the radius of convergence of the fundamental power series \eqref{Eqn=FundamentalSeries}.
If $q \not \in [\rho, \rho^{-1}]$ then $\cM_q$ is the direct sum of a II$_1$ factor and $\mathbb{C}$. For more general coxeter groups/Hecke algebras (not necessarily being right angled, or for multi-parameters $q$) this result is unknown. It deserves to be emphasized that this in particular shows that the isomorphism class of $\cM_q$ depends on $q$; an observation that was already made in the final remarks \cite[Section 19]{Davis}.

The first aim of this paper is to determine approximation properties of $\cM_q$ (assuming the same natural conditions as Garncarek). We first show that $\cM_q$ is a non-injective von Neumann algebra and therefore falls outside Connes' classification of hyperfinite factors \cite{ConnesClassification}. Secondly we show that $\cM_q$ has the weak-$\ast$ completely contractive approximation property (wk-$\ast$ CCAP). This means that there exists a net of completely contractive finite rank maps on $\cM_q$ that converges to the identity in the point $\sigma$-weak topology. In case $q = 1$ the algebra $\cM_q$ is the group von Neumann algebra of a right-angled Coxeter group. In this case the result was known. For instance the CCAP follows from Reckwerdt's result \cite{Reckwerdt} and non-injectivity follows easily from identifying a copy of the free group inside $W$. Non-injectivity can also be proved for right-angled Coxeter groups through the techniques developped in \cite{BozejkoSpeicher}.  Here we find the following:


\vspace{0.3cm}

\noindent {\bf Theorem A.} Let $q > 0$.
\begin{enumerate}
\item Let $(W,S)$ be an irreducible right-angled Coxeter system with $\vert S \vert \geq 3$.  Then $\cM_q$ is non-injective.
\item For a general right-angled Coxeter system $(W,S)$ the associated Hecke von Neumann algebra $\cM_q$  has the wk-$\ast$ CCAP  and the Haagerup property.
\end{enumerate}

The proofs of non-injectivity and the Haagerup property proceed by showing that Hecke von Neumann algebras are actually graph products \cite{CaspersFima} and then using general graph/free product techniques involving important results of Ueda \cite{Ueda}.
 For the wk-$\ast$- CCAP    we first obtain cb-estimates for radial multipliers and then use   estimates of word length projections (see Proposition \ref{Prop=CutDown}) going back to Haagerup \cite{HaagerupExample}.

\vspace{0.3cm}

Our second aim is the study of Cartan subalgebras of the Hecke von Neumann algebra $\cM_q$. Recall that a Cartan subalgebra of a II$_1$-factor is by definition a maximal abelian subalgebra whose normalizer generates the II$_1$-factor itself. Cartan subalgebras arise typically in crossed products of free ergodic probability measure preserving actions of discrete groups on a probability measure space.


In \cite{Voiculescu} Voiculescu was the first one to find factors (namely free group factors) that do not have a Cartan subalgebra. His proof relies on estimates for the free entropy dimension of the normalizer of an injective von Neumann algebra.
 Using a different approach   Ozawa and Popa \cite{OzawaPopaII} were also able to find classes of von Neumann algebras that do not have a Cartan subalgebra (including the free group factors). Ozawa and Popa actually proved that these algebras have a stronger property that afterwards became known as strong solidity: the normalizer of a diffuse injective von Neumann subalgebra generates an injective von Neumann algebra again. 
 
 After these fundamental results by Ozawa--Popa strong solidity was studied for many other von Neumann algebras. In particular in \cite{PopaVaesCrelle} Popa and Vaes (see also Chifan-Sinclair \cite{ChifanSinclair}) proved absence of Cartan subalgebras for group factors of bi-exact groups that have the CBAP. Isono \cite{IsonoExample} then put the results from \cite{PopaVaesCrelle} into a general von Neumann framework in order to prove absence of Cartan subalgebras for free orthogonal quantum groups. Isono proved that factors with the wk-$\ast$ CBAP  that satisfy condition (AO)$^+$ are strongly solid.  Using this strong solidity result by Isono we are able to prove the following.

\vspace{0.3cm}

 \noindent {\bf Theorem B.} Let $q \in [\rho, \rho^{-1}]$ with $\rho$ as in Theorem \ref{Thm=Factor}. Let $(W,S)$ be an irreducible right-angled Coxeter system with $\vert S \vert \geq 3$. Assume that $W$ is hyperbolic. Then the associated Hecke von Neumann algebra $\cM_q$ is strongly solid.

\vspace{0.3cm}

In turn as $\cM_q$ is non-injective by Theorem A we are able to derive the result announced in the title of this paper for the hyperbolic case.

\vspace{0.3cm}

 \noindent {\bf Corollary C.} Let $q \in [\rho, \rho^{-1}]$ with $\rho$ as in Theorem \ref{Thm=Factor}.  For an irreducible right-angled hyperbolic Coxeter system $(W,S)$ with $\vert S \vert \geq 3$ the associated Hecke von Neumann algebra $\cM_q$ does not have a Cartan subalgebra.

\vspace{0.3cm}

General right-angled Hecke von Neumann algebras are not strongly solid, see Remark \ref{Rmk=HyperbolicIsNecessary}. Still we can prove in some cases that they do not possess a Cartan subalgebra. We do this by showing that if $\cM_q$ were to have a Cartan subalgebra then under suitable conditions  each of the three alternatives in \cite[Theorem A]{VaesPrims} fails to be true, which leads to a contradiction.   

\vspace{0.3cm}

 \noindent {\bf Theorem D.} Let $q \in [\rho, \rho^{-1}]$.  Let $(W,S)$ be an  irreducible right-angled   Coxeter system with $\vert S \vert \geq 3$ for which the Coxeter graph satisfies the conditions of Theorem \ref{Thm=NoCartan}. Then the associated Hecke von Neumann algebra $\cM_q$ does not have a Cartan subalgebra.

\vspace{0.3cm}

\noindent {\it Structure.} In Section \ref{Sect=Prelim} we introduce Hecke von Neumann algebras and some basic algebraic properties. Lemma \ref{Lem=TExpansion} is   crucial for the results on strong solidity and the weak-$\ast$ CCAP. In Section \ref{Sect=Universal} we obtain universal properties of Hecke von Neumann algebras and prove that they decompose as graph products. We collect the consequences for Haagerup property and non-injectivity. 
  In Section \ref{Sect=Approximation} we find approximation properties of $\cM_q$ and conclude Theorem A.  Section \ref{Sect=StrongSolidity} proves the strong solidity result of Theorem B from which Corollary C shall easily follow. Finally Section \ref{Sect=Cartan} proves absence of Cartan subalgebras for the cases of Theorem D.

\vspace{0.3cm}

\noindent {\it Convention.} Let $X$ be a set and let $A,B \subseteq X$. We will briefly write $A \backslash B$ for $A \backslash (A \cap B)$.

\vspace{0.3cm}

\noindent {\bf Acknowledgements.} The author wishes to express his gratitude to the following people: Sergey Neshveyev, Lukasz Garncarek and Adam Skalski for enlightening discussions on Hecke-von Neumann algebras. The anonymous referees for several comments that led to significant improvements of the paper. 



\section{Notation and preliminaries}\label{Sect=Prelim}
Standard result on operator spaces can be found in \cite{EffrosRuan}, \cite{Pisier}. Standard references for von Neumann algebras are \cite{StratilaZsido} and \cite{Takesaki1}. Recall that {\it ucp} stands for unital completely positive.

\subsection{Coxeter groups} \label{Sect=Coxeter}
A {\it Coxeter group} $W$ is a group that is freely generated by a finite set $S$ subject to relations
\[
(s t)^{m(s,t)} = 1,
\]
for some constant $m(s,t) \in \{ 1, 2, \ldots, \infty\}$ with $m(s,t) = m(t,s) \geq 2, s\not = t$ and $m(s,s) = 1$. The constant $m(s,t)=\infty$ means that no relation is imposed, so that $s,t$ are free variables. The Coxeter group $W$ is called right-angled if either $m(s,t) = 2$ or $m(s,t) = \infty$ for all $s,t \in S, s\not = t$ and this is the only case we need in this paper. Therefore we assume from now on that $W$ is a right-angled Coxeter group with generating set $S$. The pair $(W,S)$ is also called a Coxeter system.

Let $\boldw \in W$ and suppose that $\boldw = w_1 \ldots w_n$  with $w_i \in S$. The representing expression $w_1 \ldots w_n$ is called reduced if whenever also $\boldw = w_1' \ldots w_m'$ with $w_i' \in S$ then $n \leq m$, i.e. the expression is of minimal length.  In that case we will write $\vert \boldw \vert = n$. Reduced expressions are not necessarily unique (only if $m(s,t) = \infty$ whenever $s \not = t$), but for each $\boldw \in W$ we may pick a reduced expression which we shall call minimal.

\vspace{0.3cm}

\noindent {\bf Convention:} For $\boldw \in W$ we shall write $w_i$ for the minimal representative $\boldw = w_1 \ldots w_n$.

\vspace{0.3cm}

To the pair $(W,S)$ we associate a graph $\Gamma$ with vertex set $V\Gamma = S$ and edge set $E\Gamma = \{ (s,t) \mid m(s,t) = 2 \}$. A subgraph $\Gamma_0$ of $\Gamma$ is called {\it full} if the following property holds: $\forall s,t \in V\Gamma_0$ with $(s,t) \in E\Gamma$ we have $(s,t) \in E\Gamma_0$. 

A clique in $\Gamma$ is a full subgraph in which every two vertices share an edge. We let $\Cliq(\Gamma)$ denote the set of cliques in $\Gamma$. To keep the notation consistent with the literature the empty graph is in $\Cliq(\Gamma)$ by convention (in this paper we shall sometimes exclude the empty graph from $\Cliq(\Gamma)$ explicitly or treat it as a special case to keep some of the arguments more transparent). 
 
	For $s \in S$ we set
	\[
	\Link(s) = \{ t \in S \mid m(s,t) = 2 \},
	\]
	so these are all vertices in $\Gamma$ that have distance exactly 1 to $s$. For a subset $X \subseteq V\Gamma$ we set $\Link(X) = \cap_{s \in X} \Link(s)$. We sometimes regard $\Link(X)$ as a full subgraph of $\Gamma$.

\begin{dfn}
A Coxeter system $(W,S)$ is called {\it irreducible} if the complement of $\Gamma$ is connected. Here the complement $\Gamma^c$ of the graph $\Gamma$ is the graph with the same vertex set $V\Gamma$ and for $v,w \in V\Gamma$ we have  $(v,w) \in E\Gamma^c$   if and only if $(v,w) \not \in E\Gamma$.
\end{dfn}



\subsection{Hecke von Neumann algebras}\label{Sect=SubHecke}
Let $(W,S)$ be a right-angled Coxeter system. Let $q > 0$. By \cite[Proposition 19.1.1]{Davis} there exists a unique unital $\ast$-algebra $\mathbb{C}_q(\Gamma)$ generated by a basis $\{ \widetilde{T}_\boldw \mid \boldw \in W\}$ satisfying the following relations. For every $s \in S$ and $\boldw \in W$  we have:
\[
\begin{split}
\widetilde{T}_s  \widetilde{T}_\boldw  = &
\left\{
\begin{array}{ll}
\widetilde{T}_{s\boldw} & \textrm{if } \vert s \boldw \vert > \vert \boldw \vert,\\
q \widetilde{T}_{s\boldw} + (1-q) \widetilde{T}_{\boldw} & \textrm{otherwise},\\
\end{array}
\right.\\
\widetilde{T}_\boldw^\ast   = & \widetilde{T}_{\boldw^{-1}}.
\end{split}
\]
We define normalized elements $T_{\boldw} = q^{-\vert \boldw \vert/2} \widetilde{T}_{\boldw}$. Then for $\boldw \in W$ and $s \in S$,
\begin{equation}\label{Eqn=Tmaps}
\begin{split}
T_s  T_\boldw  = &
\left\{
\begin{array}{ll}
T_{s\boldw} & \textrm{if } \vert s \boldw \vert > \vert \boldw \vert,\\
T_{s\boldw} + p T_{\boldw} & \textrm{otherwise},\\
\end{array}
\right.
\end{split}
\end{equation}
where
\[
p = \frac{q-1}{\sqrt{q}}.
\]
There is a natural positive linear tracial map $\tau$ on $\mathbb{C}_q(W)$ satisfying $\tau(T_{\boldw}) = 0, \boldw \not = 1$ and $\tau(1) = 1$. Let $L^2(\cM_q)$ be the Hilbert space given by the closure of $\mathbb{C}_q(W)$ with respect to $\langle x, y \rangle = \tau(y^\ast x)$ and let $\cM_q$ be the von Neumann algebra generated by $\mathbb{C}_q(W)$ acting on $L^2(\cM_q)$. $\tau$ extends to a state on $\cM_q$ and $L^2(\cM_q)$ is its GNS space with cyclic vector $\Omega := T_e$. $\cM_q$ is called the {\it Hecke von Neumann algebra} at parameter $q$ associated to the right-angled Coxeter system $(W, S)$.

\begin{thm}[see \cite{Garncarek})]\label{Thm=Factor}
Let $(W,S)$ be an irreducible  right-angled Coxeter system and suppose that $\vert S \vert \geq 3$. Let $\rho$ be the radius of convergence of the fundamental power series:
\begin{equation}\label{Eqn=FundamentalSeries}
\sum_{k=0}^\infty \vert \{ \boldw \in W \mid \vert \boldw \vert = k \}\vert z^{k}.
\end{equation}
For every $q \in [\rho, \rho^{-1}]$ the von Neumann algebra $\cM_q$ is a factor. For $q>0$ not in $ [\rho, \rho^{-1}]$  the von Neumann algebra $\cM_q$ is the direct sum of a factor and $\mathbb{C}$.
\end{thm}

As $\cM_q$ posesses a normal faithful tracial state the factors appearing in Theorem \ref{Thm=Factor} are of type II$_1$.

For the analysis of $\cM_q$ we shall in fact need $\cM_1$ which is the group von Neumann algebra of the Coxeter group $W$.  It can be represented on $L^2(\cM_q)$. Indeed, let $T_\boldw^{(1)}$ denote the generators of $\cM_1$ as in \eqref{Eqn=Tmaps} and let $T_\boldw$ be the generators of $\cM_q$. Set the unitary map\footnote{Unitarity follows as the vectors $T_\boldw \Omega$ are orthonormal. Indeed $\langle T_\boldw \Omega, T_\boldv \Omega \rangle = \langle T_\boldv^\ast T_\boldw \Omega,  \Omega \rangle$.
	If  $\boldv^\ast \boldw$ is reducible this expression is 0. Otherwise there exists a letter $w_1$ at the starts of $\boldv$ and $\boldw$ such that $T_\boldv^\ast T_\boldw = T_{\boldv'}^\ast T_{\boldw'} + p T_{\boldv'}^\ast T_{w_1} T_{\boldw'}$, where $w_1 \boldw' = \boldw$ and $w_1 \boldv' =\boldv$ and $\boldw'$ and $\boldv'$ are of shorter length. The term $p T_{\boldv'}^\ast T_{w_1} T_{\boldw'}$ reduces further and can be written as a sum of operators $\sum_i T_{\boldu_i}$ but each $\boldu_i$ must contain the letter $w_1$  as else $\boldw$ and $\boldv$ would not be reducible. Therefore $\langle p T_{\boldv'}^\ast T_{w_1} T_{\boldw'} \Omega, \Omega \rangle = 0$. So $\langle T_\boldv^\ast T_\boldw \Omega,  \Omega \rangle = \langle T_{\boldv'}^\ast T_{\boldw'} \Omega,  \Omega \rangle$. Continuing inductively we get $\langle T_\boldv^\ast T_\boldw \Omega,  \Omega \rangle = \delta_{\boldv, \boldw}$.}
,
\[
U: L^2(\cM_1) \rightarrow L^2(\cM_q): T_{\boldw}^{(1)} \Omega \rightarrow T_\boldw \Omega.
\]
In this paper we shall always assume that $\cM_1$ is represented on $L^2(\cM_q)$ by the identification $\cM_1 \rightarrow \cB(L^2(\cM_q)):  x \mapsto U x U^\ast$. Note that this way
\begin{equation}\label{Eqn=T1}
T_\boldv^{(1)} (T_\boldw \Omega) = T_{\boldv \boldw} \Omega.
\end{equation}
For $\boldw \in W$ we shall write $P_\boldw$ for the projection of $L^2(\cM_q)$ onto the closure of the space spanned linearly by $\{ T_\boldv \Omega \mid \vert \boldw^{-1} \boldv \vert = \vert \boldv \vert - \vert \boldw \vert  \}$ (see Remark \ref{Rmk=CreaAnni} below). For $\Gamma_0 \in \Cliq(\Gamma)$ we shall write $P_{V\Gamma_0}$ for $P_\boldw$ where $\boldw \in W$ is the product of all vertex elements of $\Gamma_0$ and $\vert V \Gamma_0 \vert$ for the number of elements in $V \Gamma_0$. Note that if $s,t \in V\Gamma_0$ then $P_s$ and $P_t$ commute and so $P_{V\Gamma_0}$ is well-defined.  Similarly we shall write $P_{\boldv V\Gamma_0}$ for $P_\boldw$ where $\boldw \in W$ is the product of $\boldv$ with all vertex elements of $\Gamma_0$.

\begin{rmk}[Creation and annihilation arguments]\label{Rmk=CreaAnni}
Note that for $\boldw, \boldv \in W$ saying that $\vert \boldw^{-1} \boldv \vert = \vert \boldv \vert - \vert \boldw \vert$ just means that the start of $\boldv$ contains the word $\boldw$. Throughout the paper we say that $s \in S$ acts by means of a creation operator on $\boldv \in W$ if $\vert s \boldv \vert = \vert \boldv \vert +1$. It acts as an annihilation operator if $\vert s \boldv \vert = \vert \boldv \vert - 1$. Note that as $W$ is right-angled we cannot have $\vert s \boldv \vert = \vert \boldv \vert$. For $\boldv, \boldw \in W$ we may always decompose $\boldw = \boldw' \boldw''$ such that $\vert \boldw \vert = \vert \boldw ' \vert + \vert \boldw '' \vert, \vert \boldw'' \boldv \vert = \vert \boldv \vert - \vert \boldw''\vert$ and $\vert \boldw \boldv \vert = \vert \boldv \vert - \vert \boldw '' \vert + \vert \boldw'\vert$. That is $\boldw$ first acts  by means of annihilations of the letters of $\boldw''$ and then $\boldw'$ acts as a creation operator on $\boldw'' \boldv$. We will use such arguments without further reference.
\end{rmk}

The following Lemma \ref{Lem=BreakDown}  together with Lemma  \ref{Lem=TExpansion} say that $T_\boldw$ decomposes in terms of a sum of operators that first act by annihilation (this is $T_{\boldu''}^{(1)}$) then a diagonal action (this is the projection $P_{\boldu V \Gamma_0}$) and finally by creation (this is $T_{\boldu'}^{(1)}$).  

\begin{dfn} \label{Dfn=Aw}
Let $\boldw \in W$.  Let $A_\boldw$ be the set of  triples $(\boldw', \Gamma_0, \boldw'')$ with $\boldw', \boldw'' \in W$ and $\Gamma_0 \in \Cliq(\Gamma)$ such that:  (1) $\boldw = \boldw' V\Gamma_0 \boldw''$, (2) $\vert \boldw \vert = \vert \boldw' \vert + \vert V\Gamma_0 \vert + \vert \boldw'' \vert$, (3)   if $s \in S$ commutes with $V\Gamma_0$ then $\vert \boldw's \vert > \vert \boldw'\vert$ (that is, letters commuting with $V\Gamma_0$ cannot occur at the end of $\boldw'$ but if they are there they should occur at the start of $\boldw''$ instead).
\end{dfn}

\begin{lem}\label{Lem=BreakDown}
For $(\boldw', \Gamma_0, \boldw'') \in A_\boldw$ there exist $\boldu, \boldu', \boldu'' \in W$
such that
\begin{equation}\label{Eqn=BreakDown}
T_{\boldw'}^{(1)} P_{V\Gamma_0} T_{\boldw''}^{(1)} =    T_{\boldu'}^{(1)}  P_{\boldu V \Gamma_0}   T_{\boldu''}^{(1)},
\end{equation}
and moreover if $s \in S$ is such that $\vert \boldu' s \vert < \vert \boldu' \vert$ then $\vert s \boldu'' \vert > \vert \boldu'' \vert$. We may assume that $\boldu' = \boldw' \boldu^{-1}$ and $\boldu'' = \boldu \boldw''$.
\end{lem}
\begin{proof}
Let $\boldu \in W$ be the (unique) element of maximal length such that $\vert \boldw'  \boldu^{-1} \vert = \vert \boldw' \vert - \vert \boldu \vert$ and $\vert \boldu \boldw'' \vert = \vert \boldw'' \vert - \vert \boldu \vert$. Set $\boldu' = \boldw' \boldu^{-1}$ and $\boldu'' = \boldu \boldw''$. It then remains to prove \eqref{Eqn=BreakDown} as the rest of the properties are obvious or follow by maximality of $\boldu$. We must show that,
\[
P_{\boldu V \Gamma_0} = T_{\boldu}^{(1)} P_{ V \Gamma_0} T_{\boldu^{-1}}^{(1)}. 
\]
Take $\boldv \in W$. If $\vert (\boldu V\Gamma_0)^{-1} \boldv \vert = \vert \boldv \vert - \vert \boldu V\Gamma_0 \vert$ (i.e. $\boldv$ starts with $\boldu V\Gamma_0$) then $\vert (V\Gamma_0)^{-1} \boldu^{-1} \boldv \vert = \vert \boldu^{-1} \boldv \vert - \vert  V\Gamma_0 \vert$ (i.e. $\boldu^{-1} \boldv$ starts with $V\Gamma_0$). We shall prove that the converse holds. First, we claim that if  $\vert (V\Gamma_0)^{-1} \boldu^{-1} \boldv \vert = \vert \boldu^{-1} \boldv \vert - \vert  V\Gamma_0 \vert$ then $\vert \boldu^{-1} \boldv \vert = \vert  \boldv \vert - \vert  \boldu^{-1} \vert$ (i.e. $\boldv$ starts with $\boldu$). Indeed,  because if this would not be the case then one of the letters in $\boldu$ would remain at the start of $\boldu^{-1} \boldv$. And as  the letters of $\boldu$ do not commute with $V\Gamma_0$ this would mean that $\vert (V\Gamma_0)^{-1} \boldu^{-1} \boldv \vert \not = \vert \boldu^{-1} \boldv \vert - \vert  V\Gamma_0 \vert$, which is a contradiction. From the initial assumption  $\vert (V\Gamma_0)^{-1} \boldu^{-1} \boldv \vert = \vert \boldu^{-1} \boldv \vert - \vert  V\Gamma_0 \vert$ ($\boldu^{-1} \boldv$ starts with $V\Gamma_0$) together with $\vert \boldu^{-1} \boldv \vert = \vert  \boldv \vert - \vert  \boldu^{-1} \vert$ ($\boldv$ starts with $\boldu$) we get that $\vert (\boldu V\Gamma_0)^{-1} \boldv \vert = \vert \boldv \vert - \vert \boldu V\Gamma_0 \vert$. 

The previous paragraph shows the first equality of
\[
P_{\boldu V \Gamma_0} (T_\boldv \Omega) =
 T_{\boldu}^{(1)} P_{ V \Gamma_0} (T_{\boldu^{-1}\boldv} \Omega)=
 T_{\boldu}^{(1)} P_{ V \Gamma_0} T_{\boldu^{-1}}^{(1)}(T_\boldv \Omega). 
\]

\end{proof}
\begin{rmk}\label{Rmk=concatenationreduced}
In Lemma \ref{Lem=BreakDown} the property that $\vert \boldu' s \vert < \vert \boldu' \vert$ implies that $\vert s \boldu''   \vert > \vert \boldu'' \vert$ is equivalent to $\vert \boldu' \boldu'' \vert = \vert \boldu' \vert + \vert \boldu'' \vert$. The words $\boldu'$ and $\boldu''$ in Lemma \ref{Lem=BreakDown} are not unique: in case $\vert s \boldu'' \vert = \vert \boldu'' \vert - 1$ and $s$ commutes with $V\Gamma_0$ then we may replace $(\boldu', \boldu'')$ by  $(\boldu's, s \boldu'')$.
\end{rmk}

\begin{lem}\label{Lem=TExpansion}
We have,
\begin{equation}\label{Eqn=TExpansion}
T_\boldw =   \sum_{(\boldw', \Gamma_0, \boldw'') \in A_\boldw} p^{\vert V \Gamma_0 \vert} T_{\boldw'}^{(1)} P_{V\Gamma_0} T_{\boldw''}^{(1)},
\end{equation}
where $A_\boldw$ is given in Definition \ref{Dfn=Aw}.
\end{lem}
\begin{proof}
The proof proceeds by induction on the length of $\boldw$. If $\vert \boldw \vert = 1$ then $T_\boldw = T_\boldw^{(1)} + p P_\boldw$ by \eqref{Eqn=Tmaps}. Now suppose that \eqref{Eqn=TExpansion} holds for all  $\boldw \in W$ with $\vert \boldw \vert = n$. Let $\boldv \in W$ be such that $\vert \boldv \vert = n+1$. Decompose $\boldv = s \boldw, \vert \boldw \vert = n, s \in S$. Then,
\begin{equation}\label{Eqn=PsTv}
\begin{split}
T_\boldv  = & T_s T_\boldw \\
= & \left( T_s^{(1)} + p P_s \right) \left(   \sum_{(\boldw', \Gamma_0, \boldw'') \in A_\boldw} p^{\vert V \Gamma_0 \vert} T_{\boldw'}^{(1)} P_{V\Gamma_0} T_{\boldw''}^{(1)}  \right) \\
= &
 \sum_{(\boldw', \Gamma_0, \boldw'')\in A_\boldw } \!\!\!\!\!\! \left( p^{\vert V\Gamma_0 \vert} T_{s \boldw'}^{(1)} P_{V\Gamma_0} T_{\boldw''}^{(1)} \right.
   \left. +   p^{\vert V\Gamma_0 \vert + 1} P_s T_{\boldw'}^{(1)} P_{V \Gamma_0} T_{\boldw''}^{(1)} \right).
\end{split}
\end{equation}
\noindent Now we need to make the following observations.

\begin{enumerate}
\item If $s \boldw' = \boldw' s$ then $P_s T_{\boldw'}^{(1)} = T_{\boldw'}^{(1)} P_s$. So in that case,
\[
 P_s T_{\boldw'}^{(1)} P_{V \Gamma_0} T_{\boldw''}^{(1)} = T_{\boldw'}^{(1)} P_s P_{V\Gamma_0} T_{\boldw''}^{(1)}.
\]
 Moreover $P_s P_{V\Gamma_0}$ equals $P_{s V\Gamma_0}$ in case $s$ commutes with all elements of $V\Gamma_0$ and it equals 0 otherwise.
\item In case $s \boldw' \not = \boldw' s$ we claim that $P_s T_{\boldw'}^{(1)} P_{V \Gamma_0} T_{\boldw''}^{(1)} = 0$. To see this, rewrite $P_s T_{\boldw'}^{(1)} P_{V \Gamma_0} T_{\boldw''}^{(1)} = P_s T_{\boldu'}^{(1)}  P_{\boldu V \Gamma_0}   T_{\boldu''}^{(1)}$ with $\boldu, \boldu', \boldu''$ as in Lemma \ref{Lem=BreakDown}.  As $s \boldw' \not = \boldw' s$ we have   $s \boldu' \not = \boldu' s$ and/or $s \boldu \not = \boldu s$ (because $\boldw' = \boldu' \boldu$ with $\vert \boldw' \vert = \vert \boldu' \vert + \vert \boldu \vert$, c.f. Lemma \ref{Lem=BreakDown}).

\begin{enumerate}
\item Assume   $s \boldu' \not = \boldu' s$. For $\boldv \in W$ with $T_{\boldu'' \boldv} \Omega$ in the range of $P_{\boldu V \Gamma_0}$, \begin{equation}\label{Eqn=Step}
P_s T_{\boldu'}^{(1)}   P_{\boldu V \Gamma_0}   T_{\boldu''}^{(1)} (T_{\boldv} \Omega ) = P_s T_{ \boldu' \boldu'' \boldv} \Omega.
\end{equation}
 Furthermore, the assertions of Lemma \ref{Lem=BreakDown} imply  $\vert \boldu' \boldu V\Gamma_0 \vert =  \vert \boldu' \vert + \vert \boldu V\Gamma_0\vert$ and therefore (recalling that $T_{\boldu'' \boldv} \Omega$ is in the range of $P_{\boldu V \Gamma_0}$) we get that $\vert \boldu' \boldu'' \boldv \vert = \vert \boldu'' \boldv \vert + \vert \boldu'\vert$ which implies  (because $s \boldu' \not = \boldu' s$ and $\boldu' \boldu'' \boldv$ starts with all letters of $\boldu'$) that \eqref{Eqn=Step} is 0. For $\boldv \in W$ with $T_{\boldu'' \boldv} \Omega$ not in the range of $P_{\boldu V \Gamma_0}$ we have $T_{\boldu'}^{(1)}   P_{\boldu V \Gamma_0}   T_{\boldu''}^{(1)} (T_{\boldv} \Omega ) = 0$. In all we conclude $P_s T_{\boldu'}^{(1)}   P_{\boldu V \Gamma_0}   T_{\boldu''}^{(1)} = 0$.
\item  Assume   $s\boldu' = \boldu's$ but $s \boldu \not = \boldu s$. Then $P_s T_{\boldu'}^{(1)} P_\boldu = T_{\boldu'}^{(1)} P_s P_\boldu = 0$.
\end{enumerate}
\end{enumerate}
\noindent So in all \eqref{Eqn=PsTv} gives,
\[
\begin{split}
T_\boldv  = & 
\sum_{(\boldw', \Gamma_0, \boldw'')\in A_\boldw } \!\!\!\!\!\!   p^{\vert V\Gamma_0 \vert} T_{s \boldw'}^{(1)} P_{V\Gamma_0} T_{\boldw''}^{(1)}   \\
&     +  \!\!\!\!\!\! \sum_{(\boldw', \Gamma_0, \boldw'')\in A_\boldw, s \boldw' = \boldw' s, s V\Gamma_0 = V\Gamma_0 s  } \!\!\!\!\!\!   p^{\vert V\Gamma_0 \vert + 1}  T_{\boldw'}^{(1)} P_{sV \Gamma_0} T_{\boldw''}^{(1)},
\end{split}
\]
and in turn an identification of all summands shows that the latter expression equals,
\[
 \sum_{(\boldv', \Gamma_0, \boldv'') \in A_{s \boldw}} p^{\vert V\Gamma_0 \vert} T_{\boldv'}^{(1)} P_{V\Gamma_0} T_{\boldv''}^{(1)}.
\]
This concludes the proof.
\end{proof}

\subsection{Group von Neumann algebras} Let $\sG$ be a discrete group with left regular representation $s \mapsto \lambda_s$ and group von Neumann algebra $\cL(\sG) = \{ \lambda_s \mid s \in \sG \} ''$. We let $A(\sG)$ be the Fourier algebra consisting of functions $\varphi(s) = \langle \lambda_s \xi, \eta \rangle, \xi, \eta \in \ell^2(\sG)$. There is a pairing between $A(\sG)$ and $\cL(\sG)$ which is given by $\langle \varphi, \lambda(f) \rangle = \int_\sG f(s) \varphi(s) ds$ which turns $A(\sG)$ into an operator space that is completely isometrically identified  with $\cL(\sG)_\ast$. We let $M_\CB A(\sG)$ be the space of completely bounded Fourier multipliers of $A(\sG)$. For $m \in M_\CB A(\sG)$ we let $T_m: \cL(\sG) \rightarrow \cL(\sG)$ be the normal completely bounded map determined by $\lambda(f) \mapsto \lambda(mf)$. The following theorem is due to Bozejko and Fendler \cite{BozejkoFendler} (see also \cite[Theorem 4.5]{JungeNeufangRuan}).
\begin{thm}\label{Thm=BozejkoFendler}
Let $m \in M_{\CB}A(\sG)$. There exists a unique normal completely bounded map $M_m: \cB(\ell^2(\sG)) \rightarrow \cB(\ell^2(\sG))$ that is an $L^\infty(\sG)$-bimodule homomorphism and such that $M_m$ restricts to $T_m: \lambda(f) \mapsto \lambda(mf)$ on $\cL(\sG)$. Moreover, $\Vert M_m \Vert_{\CB} = \Vert T_m \Vert_{\CB} = \Vert m \Vert_{M_{\CB}A(\sG)}$.
\end{thm}

The map $M_m$ is called the Herz-Schur multiplier.

\section{Universal property and conditional expectations}\label{Sect=Universal}

In this section we establish  universal properties for $\cM_q$ and consequently show that $\cM_q$ is non-injective and has the Haagerup property. 

\subsection{Universal properties} 

\begin{thm}\label{Thm=Universal}
Let $q>0$ put $p = (q-1)/\sqrt{q}$ and let  $(W, S)$ be a right-angled Coxeter system with associated Hecke von Neumann algebra $(\cM_q, \tau)$. Suppose that $(\cN,\tau_\cN)$ is a von Neumann algebra with GNS faithful state $\tau_\cN$ that is generated by self-adjoint operators $R_s, s \in S$ that satisfy the relations $R_s R_t = R_t R_s$ whenever $m(s,t) = 2$, $R_s^2 = 1 + p R_s, s \in S$ and further $\tau_\cN( R_{w_1} \ldots R_{w_n}) = 0$ for every non-empty reduced word $\boldw = w_1 \ldots w_n \in W$. Then there exists a unique normal $\ast$-homomorphism $\pi: \cM_q \rightarrow \cN$ such that $\pi(T_s) = R_s$. Moreover $\tau_\cN \circ \pi = \tau$.
\end{thm}
\begin{proof}
The proof is routine, c.f. \cite[Proposition 2.12]{CaspersFima}. We sketch it here. Let $(L^2(\cN), \pi_\cN, \eta)$ be a GNS construction for $(\cN, \tau_{\cN})$. As $\tau_\cN$ is GNS faithful we may assume that $\cN$ is represented on $L^2(\cN)$ via $\pi_\cN$. We define a linear map $V: L^2(\cM_q) \rightarrow L^2(\cN)$ by $V \Omega = \eta$ and
\[
V (T_\boldw \Omega) = R_\boldw \eta, \qquad \textrm{ where } \boldw \in W,
\]
and $R_\boldw := R_{w_1} \ldots R_{w_n}$.  One checks that $V$ is  isometric by showing that  $\{ R_\boldw \eta \mid \boldw \in W \}$ is an orthonormal system.\footnote{The proof goes as follows. We may find unique coefficients $c_\boldv$ such that
\[
T_{w_n'} \ldots T_{w_1'} T_{w_1} \ldots T_{w_n} = \sum_{\boldv \in W} c_{\boldv} T_{\boldv}.
\]
We have $c_\emptyset = 1$ if $\boldw = \boldw'$ and $c_\emptyset = 0$ if $\boldw \not = \boldw'$ by comparing the trace of both sides of this expression.
In fact the coefficients $c_\boldv$ may be found by using the commutation relations for $T_s$ and the Hecke relation $T_s^2 = 1 + p T_s$ to `reduce' the left hand side of this expression. As the same relations hold for the operators $R_s$ (by assumption of the lemma) we also get $R_{w_n'} \ldots R_{w_1'} R_{w_1} \ldots R_{w_n} = \sum_{\boldv \in W} c_{\boldv} R_{\boldv}$. So,
\[
\begin{split}
 \langle R_\boldw \eta, R_{\boldw'} \eta \rangle
= \tau_{\cN} (  R_{\boldw'}^\ast R_\boldw )
=  \tau_{\cN}( R_{w_n'} \ldots R_{w_1'} R_{w_1} \ldots R_{w_n}    )
= \tau_{\cN}\left(  \sum_{\boldv \in W} c_{\boldv} R_{\boldv} \right) = c_\emptyset.
\end{split}
\]
This proves that indeed $V$ is isometric.} Putting $\pi( \: \cdot \: ) = V (\: \cdot \:) V^\ast$ concludes the lemma. As $V \Omega = \eta$ we get $\tau_\cN \circ \pi = \tau$.
\end{proof}
\begin{rmk}
 Note that the property $T_s^2 = 1 + p T_s, s \in S$ with $p = \frac{q-1}{\sqrt{q}}$ is equivalent to the usual Hecke relation
  $ (\sqrt{q}\: T_s - q) (\sqrt{q} \: T_s + 1) = 0$ that appears in the literature.
 \end{rmk}

We shall say that $(\widetilde{W}, \widetilde{S})$ is a Coxeter subsystem of $(W,S)$ if $\widetilde{S} \subseteq S$ and $\widetilde{m}(s,t) = m(s,t)$ for all $s,t \in \widetilde{S}$. Here $\widetilde{m}$ is the function on $\widetilde{S} \times \widetilde{S}$ that determines the commutation relations for $\widetilde{W}$, c.f. Section \ref{Sect=Coxeter}.

\begin{cor}\label{Cor=Expected}
Let $q > 0$. Let $(\widetilde{W}, \widetilde{S})$ be a Coxeter subsystem of a right-angled Coxeter system $(W, S)$. Let  $\widetilde{\cM}_q$ and $\cM_q$ be their respective Hecke von Neumann algebras.  Then naturally $\widetilde{\cM}_q$ is a von Neumann subalgebra of $\cM_q$. In particular, there exists a trace preserving normal conditional expectation $\mathcal{E}: \cM_q \rightarrow \widetilde{\cM}_q$.
\end{cor}
\begin{proof}
Theorem \ref{Thm=Universal} implies that $\widetilde{\cM}_q$ is a von Neumann subalgebra of $\cM_q$ and the canonical trace of $\cM_q$ agrees with the one on $\widetilde{\cM}_q$. Therefore $\widetilde{\cM}_q$ admits a trace preserving normal conditional expectation value, c.f. \cite[Theorem IX.4.2]{Takesaki2}.
\end{proof}

Consider the Hecke von Neumann algebra $\cM_q$ for the case that $S$ is a one-point set, $q> 0$ and $p = \frac{q-1}{\sqrt{q}}$. In that case we have $W = \{ e, s\}$ and $L^2(\cM_q)$ has a canonical basis $\Omega$ and $T_s \Omega$. With respect to this basis $T_s$ takes the form
$
\left(
\begin{array}{cc}
0 & 1 \\
1 & p \\
\end{array}
\right)
$
and one sees (using for example the relation $T_s^2 = 1 + p T_s$) that $\cM_q = \mathbb{C} {\rm Id}_2 \oplus \mathbb{C} T_s$, i.e. it is two dimensional. The following corollary uses the graph product, for which we refer to \cite{CaspersFima}. It is a generalization of the free product by adding a commutation relation to vertex algebras that share an edge; the free product is then given by a graph product over a graph with no edges. In \cite{CaspersFima} the symbol $\ast$ was used for graph products.   We use the notation $\star$ instead to distinguish them from free (amalgamated) products. 

\begin{cor}\label{Cor=GraphDec}
Let $(W,S)$ be an arbitrary right-angled Coxeter system and let $q> 0$. Let $\Gamma$ be the graph associated to $(W,S)$ as before. For $s \in S$ let $\cM_q(s)$ be the 2-dimensional Hecke von Neumann subalgebra corresonding to the one-point set $\{ s \}$.  Then we have a graph product decomposition $\cM_q = \oast_{s \in V\Gamma}  \cM_q(s)$.
\end{cor}
\begin{proof}
Let $T_s \in \cM_q, s \in S$ be the operators as introduced in Section \ref{Sect=SubHecke}. Let $\widetilde{T}_s, s \in S$ be the operator $T_s$ but then considered in the algebra $\cM_q(s)$ which in turn is contained in $\oast_{s \in V\Gamma}   \cM_q(s)$ with conditional expectation. Now the map $T_s \mapsto \widetilde{T}_{s}$ determines an isomorphism by  Theorem \ref{Thm=Universal} and the universal property of the graph product given by  \cite[Proposition 2.12]{CaspersFima}.
\end{proof}


\subsection{Non-injectivity}

\begin{dfn}
	A von Neumann algebra $\cM \subseteq \cB(\cH)$ is called injective if there exists a conditional expectation $\mathcal{E}: \cB(\cH) \rightarrow \cM$. 
\end{dfn}

\begin{thm}\label{Thm=NonInjective3}
	Let $(W,S)$ be an irreducible right-angled Coxter system with $\vert S \vert \geq 3$. Then $\cM_q$ is non-injective. 
\end{thm}
\begin{proof}
	It suffices to prove that $\cM_q$ contains an expected non-injective von Neumann subalgebra. Now any irreducible Coxeter system $(W,S)$ contains a Coxeter subsystem $(\widetilde{W}, \widetilde{S})$ either of the form $\widetilde{S} = \{ r,s,t \}$ with $\widetilde{m}(r,s) = \widetilde{m}(r,t) = \widetilde{m}(s,t) = \infty$ or $\widetilde{S} = \{ r,s,t \}$ with $\widetilde{m}(r,s) = \widetilde{m}(r,t) = \infty$ and $\widetilde{m}(s,t) = 2$. So it satisfies to prove non-injectivity for these systems. In both cases, for $q$ fixed, set $\cM$ to be the Hecke von Neumann algebra of the Coxter system consisting of just $\{ r \}$. $\cM$ has dimension 2.  Set $\cN$ to be the Hecke von Neumann algebra of the Coxter system $\{s, t\}$, which is infinite dimensional in case $m(s,t) = \infty$ and 4 dimensional if $m(s,t) = 2$ (being the tensor product of two 2 dimensional algebras).  Then $\cM_q$ is isomorphic to the free product $\cM \ast \cN$ over the canonical traces by Corollary \ref{Cor=GraphDec} and \cite[Remark 3.23]{CaspersFima}.  As  $\dim(\cM) + \dim(\cN) \geq 5$ it follows that $\cM_q$ is non-injective from \cite[Theorem 4.1]{Ueda} (see comment (5) in \cite[Remark 4.2]{Ueda}).
\end{proof}

\subsection{Haagerup property}

We first construct radial multipliers.

\begin{prop}\label{Prop=Radial}
	Let $(W,S)$ be a right-angled Coxeter group with Hecke von Neumann algebra $\cM_q, q > 0$.
	For every $0 < r < 1$ there exist a normal unital completely positive map $\Phi_r: \cM_q \rightarrow \cM_q$ that is determined by $\Phi_r(T_{\boldw}) = r^{\vert \boldw \vert} T_{\boldw}$.
\end{prop}
\begin{proof}
	As in Corollary \ref{Cor=GraphDec} we identify $\cM_q$ with the graph product $\ast_{s \in V\Gamma}^\Gamma (\cM_q(s), \tau_{s})$ where $\tau_s$ is the tracial state on $\cM_q(s)$. Consider the map $\Phi_{r,s}: \cM_q(s) \rightarrow \cM_q(s)$ determined by $1 \mapsto 1, T_s \mapsto r T_s$. This map is unital and completely postivie: indeed consider matrices
	\[
	A :=
	\left(
	\begin{array}{cc}
	\sqrt{1 - r} & 0 \\
	0 & 0
	\end{array}
	\right),
	B =
	\left(
	\begin{array}{cc}
	0 &  \sqrt{1 - r} \\
	0 & 0
	\end{array}
	\right),
	C =
	\left(
	\begin{array}{cc}
	\sqrt{r} &  0 \\
	0 & \sqrt{r}
	\end{array}
	\right).
	\]
	Then $\Phi_{r,s}$ agrees with $x \mapsto A^\ast x A + B^\ast x B + C^\ast x C$ as before Corollary \ref{Cor=GraphDec}  we  already noted that $
	T_s =
	\left(
	\begin{array}{cc}
	0 &  1 \\
	1 &  p
	\end{array}
	\right)$. Furthermore $\Phi_{r,s}$ preserves the trace $\tau_s$ as $\tau_s$ is the vector state associated with $(1, 0)^t$. Therefore we may apply \cite[Proposition 2.30]{CaspersFima} and obtain the graph product ucp map $\Phi_{r} := \oast_{s \in V\Gamma} \Phi_{r,s}$ which proves the proposition.
\end{proof}

\begin{dfn}
	Recall that a von Neumann algebra $\cM$ with normal faithful tracial state $\tau$ has the {\it Haagerup property} if there exists a net $\Phi_i$ of $\tau$-preserving ucp maps $\cM \rightarrow \cM$ such that $T_i: x \Omega_\tau \mapsto \Phi_i(x) \Omega_\tau$ is compact and converges to 1 strongly.
\end{dfn}

\begin{thm}
For any Coxeter system $(W,S)$ and any $q>0$ the von Neumann algebra $\cM_q$ has the Haagerup property. 
\end{thm}
	\begin{proof}	
		If $S$ is finite  Proposition \ref{Prop=Radial} directly shows that $\cM_q$ has the Haagerup property by letting $r \nearrow 1$. Then the general case follows by an inductive limit argument on finite Coxeter subsystems using the conditional expectations from Corollary \ref{Cor=Expected}.
\end{proof}

\section{Completely contractive approximation property}\label{Sect=Approximation}

We show that for a right angled Coxeter system $(W,S)$ the Hecke von Neumann algebra $\cM_q$ has the wk-$\ast$ CCAP, see Definition \ref{Dfn=CBAP}. The proof follows a -- by now standard -- strategy of Haagerup \cite{HaagerupExample} by considering radial multipliers first and then showing that word length cut-downs have a complete bound that is at most polynomial in the word length.

\subsection{Creation/annihilation arguments}\label{Sect=CreationAnnihilation}\hyphenation{anni-hi-la-tion}
Here we present some combinatorical arguments that we need in Section \ref{Sect=CutDown}. We have chosen to separate these from the proofs of Section \ref{Sect=CutDown} so that the reader could skip them at first sight.

  We introduce the following notation. Let $\boldx, \boldw \in W$. We shall write $\boldw \leq \boldx$ for saying that $\vert \boldw^{-1} \boldx \vert = \vert \boldx \vert - \vert \boldw \vert$.   Then $\boldw < \boldx$ is defined naturally. So $\boldw \leq \boldx$ means that $\boldw$ is obtained from $\boldx$ by cutting off a tail. An element $\boldv \in W$ is called a {\it clique word} in case its letters form a clique. For $\Lambda$  a clique in $W$ and $\boldv \in W$ we define $\boldv(2, \emptyset)$ as the maximal\footnote{Suppose that $\Gamma_0$ and $\Gamma_1$ are cliques such that both $\vert \boldv V \Gamma_i \vert = \vert \boldv \vert - \vert V \Gamma_i \vert$ then the letters $V\Gamma_0$ and $V\Gamma_1$ must commute. So  the union $\Gamma_2 = \Gamma_0 \cup \Gamma_1$ is a clique with $\vert \boldv V \Gamma_2 \vert = \vert \boldv \vert - \vert V \Gamma_2 \vert$.} clique $\Gamma_0$ such that $\vert \boldv V \Gamma_0 \vert = \vert \boldv \vert - \vert V \Gamma_0 \vert$. Then we set
 the decomposition $\boldv = \boldv(1, \Lambda)  \boldv(2, \Lambda)$ with $\vert \boldv \vert = \vert \boldv(1, \Lambda) \vert + \vert \boldv(2, \Lambda) \vert$ and  $\boldv(2, \Lambda) = \boldv(2, \emptyset) \backslash \Lambda$ (which uniquely determines $\boldv(1, \Lambda)$).  For $\boldg \leq \boldx$ we let $\Lambda_{\boldg, \boldx}$ be  $(\boldx^{-1} \boldg)(2, \emptyset)$. In other words $\Lambda_{\boldg, \boldx}$ is the maximal clique that appears at the start of $\boldg^{-1} \boldx$.
 We let $C(\boldg, \boldx)$ be the collection of $\boldw \in W$ with $\boldg \leq \boldw \leq \boldg \Lambda_{\boldg, \boldx}$. Note that $C(\boldg, \boldx)$ contains at least $\boldg$ and $\boldg \Lambda_{\boldg, \boldx}$ (and the latter elements can be equal). We write $C(\boldg, +)$ for $\cup_{\boldg \leq \boldx} C(\boldg, \boldx)$.

 \begin{example}
 Consider the Coxeter system $(W,S)$ with $S = \{ r,s,t \}$ in which $m(r,s) = 2$ and $m(r,t) = m(s,t) = \infty$. Consider $\boldv = trs$. Then $\boldv(1, \emptyset) = t,  \boldv(2, \emptyset) = rs, \boldv(1, r) = tr$ and $\boldv(2, r) = s$. Also $\Lambda_{t, trst} = \{t, tr, ts, trs  \}$.
 \end{example}

\begin{lem}\label{Lem=ExclusionThingy}
Let $\boldx, \boldw \in W$. Let $\boldw = \boldw' \boldw''$ be the decomposition with $\vert \boldw \vert = \vert \boldw' \vert + \vert \boldw'' \vert$ such that $\vert \boldw'' \boldx \vert = \vert \boldx \vert - \vert \boldw''\vert$  and $\vert \boldw \boldx \vert = \vert \boldx \vert - \vert \boldw'' \vert + \vert \boldw' \vert$. Take $(\boldw'')^{-1} \leq \boldg \leq \boldx$. Then, for $\boldv \in C(\boldg, \boldx)$,
\begin{equation}\label{Eqn=LaChouffe3}
(\boldw\boldv)(2, (\boldw \boldg)(2, \emptyset) \backslash \boldg(2, \emptyset) ) = \boldv(2, \boldg(2, \emptyset) \backslash (\boldw\boldg)(2, \emptyset) )
\end{equation}
and
 \begin{equation}\label{Eqn=LaChouffe4}
\begin{split}
 \vert (\boldw \boldv)(1, (\boldw\boldg)(2, \emptyset) \backslash \boldg(2, \emptyset) ) \vert = \vert \boldv(1, \boldg(2, \emptyset) \backslash \boldw\boldg(2, \emptyset)) \vert- \vert \boldw'' \vert + \vert \boldw' \vert.
\end{split}
\end{equation}
\end{lem}
\begin{proof}
Let $\boldv \in C(\boldg, \boldx)$. The clique $\boldv(2, \emptyset)$ consists of the clique $\boldg^{-1} \boldv$ plus all letters in $\boldg(2, \emptyset)$ that commute with  $\boldg^{-1} \boldv$. Therefore $\boldv(2, \boldg(2, \emptyset) \backslash (\boldw \boldg)(2, \emptyset) )$ is the clique consisting of $\boldg^{-1}\boldv$ plus all letters in $(\boldw \boldg)(2, \emptyset) \cap \boldg(2, \emptyset)$ that commute with $\boldg^{-1} \boldv$. On the other hand $(\boldw \boldv)(2, \emptyset)$ consists of the clique $\boldg^{-1}\boldv$ together with all letters in $(\boldw \boldg)(2, \emptyset)$ that commute with $\boldg^{-1} \boldv$. Then $(\boldw \boldv)(2, (\boldw \boldg)(2, \emptyset) \backslash \boldg(2, \emptyset) )$ equals $\boldg^{-1} \boldv$ together with all elements in $(\boldw \boldg)(2, \emptyset) \cap \boldg(2, \emptyset)$ that commute with $\boldg^{-1} \boldv$. So we conclude \eqref{Eqn=LaChouffe3}.  Therefore,
\begin{equation}
\begin{split}
 \vert (\boldw \boldv)(1, (\boldw\boldg)(2, \emptyset) \backslash \boldg(2, \emptyset)) \vert
= & \vert \boldw \boldv \vert - \vert (\boldw \boldv)(2, (\boldw\boldg)(2, \emptyset) \backslash \boldg(2, \emptyset))\vert \\
= & \vert \boldv \vert - \vert \boldw'' \vert + \vert \boldw' \vert
- \vert \boldv(2, \boldg(2, \emptyset) \backslash (\boldw\boldg)(2, \emptyset)) \vert \\
= & \vert \boldv(1, \boldg(2, \emptyset) \backslash \boldw\boldg(2, \emptyset)) \vert- \vert \boldw'' \vert + \vert \boldw' \vert.
\end{split}
\end{equation}

\end{proof}

\begin{lem}\label{Lem=Tediously}
Let $\boldx, \boldw \in W$ and decompose $\boldw = \boldw' \boldw''$ such that $\vert \boldw \vert = \vert \boldw'\vert + \vert \boldw''\vert, \vert \boldw'' \boldx \vert = \vert \boldx \vert - \vert \boldw'' \vert$ and $\vert \boldw \boldx \vert = \vert \boldx \vert - \vert \boldw'' \vert + \vert \boldw' \vert$. Let $(\boldw'')^{-1} \leq \boldg \leq \boldx$. Then:
\begin{enumerate}
\item $\boldg(2, \emptyset) \backslash (\boldw \boldg)(2, \emptyset) = \boldg(2, \emptyset) \backslash (\boldw'' \boldg)(2, \emptyset)$;
\item For $\boldv \in C(\boldg, \boldx)$ we have
\begin{equation}\label{Eqn=EndSet}
\boldv(2, \boldv(2, \emptyset) \backslash (\boldw'' \boldv)(2, \emptyset) ) = \boldv(2, \boldg(2, \emptyset) \backslash (\boldw'' \boldg)(2, \emptyset)  ).
\end{equation}
\end{enumerate}
\end{lem}
\begin{proof}
(1) Because $(\boldw'')^{-1} \leq \boldg \leq \boldx$ we also have $\vert \boldw'' \boldg \vert = \vert \boldg \vert - \vert \boldw'' \vert$ and $\vert \boldw \boldg \vert = \vert \boldg \vert - \vert \boldw'' \vert + \vert \boldw' \vert$. So $\boldw'$ creates letters in $\boldw'' \boldg$ so that  $\boldg(2, \emptyset) \backslash (\boldw \boldg)(2, \emptyset) = \boldg(2, \emptyset) \backslash (\boldw'' \boldg)(2, \emptyset)$.

(2)  Let $A$ be the set of letters in $\boldg(2, \emptyset)$ that commute with $\boldg^{-1}\boldv$. The clique $\boldv(2, \emptyset)$ consists of $\boldg^{-1} \boldv \cup A$  . This means that $\boldv(2, \boldv(2, \emptyset) \backslash (\boldw'' \boldv)(2, \emptyset) )$ consists of $\boldg^{-1} \boldv \cup A$ intersected with $(\boldw'' \boldv)(2, \emptyset)$. The intersection of $(\boldw'' \boldv)(2, \emptyset)$ with $\boldg^{-1} \boldv$ is $\boldg^{-1} \boldv$   so that $\boldv(2, \boldv(2, \emptyset) \backslash (\boldw'' \boldv)(2, \emptyset) ) = \boldg^{-1} \boldv \cup (A \cap (\boldw''\boldv)(2, \emptyset))$. On the other hand $\boldv(2, \boldg(2, \emptyset) \backslash (\boldw'' \boldg)(2, \emptyset) )$ equals $\boldg^{-1} \boldv \cup (A \cap (\boldw''\boldg)(2, \emptyset))$ and as $\boldg(2, \emptyset) \cap (\boldw'' \boldg)(2, \emptyset) = \boldg(2, \emptyset) \cap (\boldw'' \boldv)(2, \emptyset)$  clearly  $(A \cap (\boldw''\boldv)(2, \emptyset)) = (A \cap (\boldw''\boldg)(2, \emptyset))$. This proves \eqref{Eqn=EndSet}.

\end{proof}

 Although Coxeter groups generally do not have polynomial growth (nor they are hyperbolic) we still have the polynomial estimate of the following Lemma \ref{Lem=PolynomialBound}.   We do not believe that the degree of the polynomial bound we obtain in Lemma \ref{Lem=PolynomialBound} is optimal, but it suffices for our purposes and it admits a short proof.

\begin{lem}\label{Lem=PolynomialBound}
Let $W$ be a right-angled Coxeter group with finite graph $\Gamma$.
Let $\boldx \in W$. For $a \in \mathbb{N}$ define
\[
\kappa_{\boldx}(a) = \vert \left\{ \boldw \leq \boldx \mid \vert \boldw \vert = a \right\} \vert.
\]
Then $\kappa_{\boldx}(a) \leq C a^{\vert V \Gamma \vert -2}$. Moreover, the constant $C$ can be taken uniformly in $\boldx$.
\end{lem}
\begin{proof}
To do the proof we shall actually count a more refined number. We write $\Lambda \leq \Gamma$ for saying that $\Lambda$ is a complete subgraph of $\Gamma$. We say that $\boldw$ is a $(\leq \Lambda)$-word if its letters (in reduced form) are all in $V\Lambda$ (they do not need to exhaust all of $V\Lambda$); we say that $\boldw$ is a $\Lambda$-word if its letters are exactly $V\Lambda$. Then define
\begin{equation}\label{Eqn=CountKappa}
\kappa_{\boldx}^{\Lambda}(a) = \vert \left\{ \boldv \leq \boldx \mid \vert \boldv \vert = a \textrm{ and } \boldv  \textrm{ is a } (\leq\Lambda)\textrm{-word} \right\} \vert.
\end{equation}

 Let $c$ and $k_0$  be constants such that for $a \in \{ 0, 1\}$ we have for all $\emptyset \not = \Lambda < \Gamma$ we have $ \kappa^{\Lambda}_\boldx(a) \leq c (a+k_0)^{\vert V \Lambda \vert -2 }$ and further for all $a \in \mathbb{N}$ and all non-empty complete subgraphs $\Lambda$ of $\Gamma$ we have  $2^{\vert V \Lambda \vert} ca  \leq  (a+k_0)^{2}$. We prove by induction on $a\in \mathbb{N}$ that for all  $\emptyset \not = \Lambda < \Gamma$ we have $\kappa_{\boldx}^{\Lambda}(a)  \leq c (a +k_0)^{\vert V\Lambda \vert - 2}$. 
 
 \vspace{0.3cm}
 
\noindent  {\it Inductive step.} Pick some fixed $\boldw < \boldx$ with $\vert \boldw \vert = a$ and $\boldw$ a $\Lambda$-word. Now if $\boldv < \boldx$ with $\vert \boldv \vert = a$ then let $\boldv_0$  be an element of maximal length such that both $\boldv_0 \leq \boldv$ and $\boldv_0 \leq \boldw$ (we leave in the middle if $\boldv_0$ is unique). 
 
 Let $s \in S$ be a letter that appears at the start of $\boldv_0^{-1} \boldw$. 
 We claim that  the letter $s$ must commute with $\boldv_0^{-1} \boldv$. Indeed, first observe that
as $\boldv_0$ has maximal length $s$ cannot appear at the start of $\boldv_0^{-1} \boldv$.  
 Further, write $\boldx = \boldv_0 (\boldv_0^{-1} \boldv) (\boldv^{-1} \boldx)$ and $\boldx = \boldv_0 (\boldv_0^{-1} \boldw) (\boldw^{-1} \boldx)$. So, 
 \begin{equation}\label{Eqn=MoveS}
 (\boldv_0^{-1} \boldw) (\boldw^{-1} \boldx) = (\boldv_0^{-1} \boldv) (\boldv^{-1} \boldx).
 \end{equation}
  $s$ appears at the start of  $(\boldv_0^{-1} \boldw)$ and hence this letter must occur somewhere in the expression $(\boldv_0^{-1} \boldv) (\boldv^{-1} \boldx)$ as well. Consider the first occurence of $s$ in  $(\boldv_0^{-1} \boldv) (\boldv^{-1} \boldx)$. All the letters before it must then commute with $s$ as otherwise the equality \eqref{Eqn=MoveS}, telling that $s$ is at the start,  is violated (c.f. the normal form theorem \cite[Theorem 3.9]{Green}). But then $s$ does not occur on $\boldv_0^{-1} \boldv$ as then it is automatically at its start. So the first time $s$ occurs in  $(\boldv_0^{-1} \boldv) (\boldv^{-1} \boldx)$ is in the part  $(\boldv^{-1} \boldx)$ and so it commutes with all elements in  $(\boldv_0^{-1} \boldv)$.

So if $\boldv_0^{-1} \boldw$ is a $\Lambda$-word then $\boldv_0^{-1} \boldv$ is a $\Link(\Lambda)$-word (recall $\Link(\Lambda) = \cap_{s \in V\Lambda} \Link(s)$); in fact it must be a $(\Link(\Lambda) \cap \Lambda)$-word as we only deal with words with letters in $\Lambda$. Moreover $\boldv_0^{-1} \boldv$ must appear at the start of $\boldw^{-1} \boldx$.  So every word in the set we count in \eqref{Eqn=CountKappa} is obtained from $\boldw$ by cutting off a tail (this is $\boldv_0^{-1} \boldw$) and then adding a tail of the same size with commuting letters (this is $\boldv_0^{-1} \boldv$). This certainly gives the inequality,
 \[
 \kappa_\boldx^{\Lambda}(a) \leq \sum_{\Lambda' \leq \Lambda} \sum_{\begin{array}{c} \boldv \leq \boldw \\ \boldv^{-1} \boldw \textrm{ is a } \Lambda'\textrm{-word} \end{array}}    \kappa_{\boldw^{-1} \boldx}^{\Link(\Lambda')} (\vert  \boldv^{-1} \boldw \vert).
 \]
 Note that the number of $\boldv \in W$ with  $\boldv < \boldw, \vert \boldv \vert = l$ and $\boldv$ a $\Lambda'$ -word it is smaller than or equal to $\kappa_{\boldw^{-1}}^{\Lambda'}(\vert \boldw \vert - l)$.  In case $l = 0$ we have $\kappa_{\boldw^{-1}}^{\Lambda'}(\vert \boldw \vert - l) = 1$ (elementary) and in case $l > 0$ we can apply our induction hypothesis to get $\kappa_{\boldw^{-1}}^{\Lambda'}(\vert \boldw \vert - l) \leq c (a - l + k_0)^{\vert V \Lambda' \vert - 2}$. Therefore we get,
\[
\begin{split}
  \kappa_{\boldx}^{\Lambda}(a)
\leq & \sum_{\Lambda' < \Lambda} \sum_{ \begin{array}{c} \boldv < \boldw \\ \boldv^{-1} \boldw \textrm{ is a } \Lambda'\textrm{-word} \end{array} }  c (a + k_0)^{\vert \Link(\Lambda') \cap V\Lambda \vert - 2} \\
\leq & \sum_{\Lambda' < \Lambda} \sum_{l = 0}^{a} c^2 (a-l+k_0)^{\vert V \Lambda' \vert -2} (a + k_0 )^{\vert \Link(\Lambda') \cap V\Lambda \vert - 2} \\
\leq & \sum_{\Lambda' < \Lambda} \sum_{l = 0}^{a} c^2 (a+k_0)^{\vert V \Lambda' \vert -2} (a + k_0 )^{\vert \Link(\Lambda') \cap V\Lambda \vert - 2} 
\end{split}
\]
Since the intersection of each  $V\Lambda'$ and $\Link(\Lambda') \cap V \Lambda$ is empty we find,
\[
\begin{split}
\kappa_{\boldx}^{\Lambda}(a) \leq & \sum_{\Lambda' < \Lambda} \sum_{l = 0}^{a} c^2 (a+k_0)^{\vert V \Lambda \vert -4}  \\
\leq & 2^{\vert V \Gamma \vert} c^2  (a+1)  (a+k_0)^{\vert V \Lambda \vert -4}\\
\leq & c (a+k_0)^{\vert V\Lambda \vert -2}.
\end{split}
\]
The last line follows from the choice of $c$ and $k_0$.

\end{proof}

\subsection{Word length projections}\label{Sect=CutDown}
The aim of this section is to prove that $T_\boldw \mapsto \delta(\vert \boldw \vert \leq n) T_\boldw$ gives a complete bounded multiplier of $\cM_q$ with complete bound growing at most polynomially in $n$. Firstly we simplify notation a little bit.

\begin{rmk}
Note that we may identify $\ell^2(W)$ with basis $\delta_\boldx, \boldx \in W$ with $L^2(\cM_q)$ with basis $T_\boldx \Omega$. This way $T_\boldw^{(1)}$ acts on $\ell^2(W)$ by means of the left regular representation.
\end{rmk}

 We borrow the following construction from \cite{OzawaWeak}. We let $B_f(W)$ be the set of finite subsets of $W$. For $A \in B_f(W)$ we define $\widetilde{\xi}_A^{\pm}$ to be the vectors in $\ell^2(B_f(W))$ given by
 \[
 \widetilde{\xi}_A^{+}(\omega) = \left\{
\begin{array}{ll}
 1 & \textrm{if } \omega \subseteq A, \\
 0 & \textrm{otherwise},
\end{array}
 \right.
\qquad
 \widetilde{\xi}_A^{-}(\omega) = \left\{
 \begin{array}{ll}
 (-1)^{\vert \omega \vert} & \textrm{if } \omega \subseteq A, \\
 0 & \textrm{otherwise},
 \end{array}
 \right.
 \]
Using the binomial formula (see Lemma 4 of \cite{OzawaWeak}), we have $\Vert \widetilde{\xi}_A^{\pm}  \Vert^2 = 2^{\vert A \vert}$ and
\[
\langle \widetilde{\xi}_A^{+}, \widetilde{\xi}_B^{-} \rangle = \left\{
\begin{array}{ll}
0 & A \cap B \not = \emptyset,\\
1 & \textrm{otherwise}.
\end{array}
\right.
\]
 We let
 \begin{equation}\label{Eqn=Rspan}
 \mathcal{R} = {\rm span} \left\{ P_\boldw \mid \boldw \in W \right\}.
  \end{equation}
  Let $Q_\boldw$ be the operator
  \[
  Q_\boldw \delta_\boldx = \delta(\boldw = \boldx) \delta_\boldx,
   \]
   i.e. $Q_\boldw$ is the Dirac delta function at $\boldw$ seen as a multiplication operator.

 \begin{lem}\label{Lem=Aux2}   For $\boldw \in W$ we have $Q_\boldw  =  \sum_{\boldv \in C(\boldw, +)} (-1)^{\vert \boldw^{-1} \boldv \vert} P_\boldv $.
 \end{lem}
 \begin{proof}
 Firstly, $Q_\boldw(\boldw)  =  1  = P_\boldw(\boldw) =  ( \sum_{\boldv \in C(\boldw, +)}  (-1)^{\vert \boldw^{-1} \boldv \vert}  P_\boldv  )(\boldw)$.  Let $\boldx \in W$. If $\boldw \not \leq \boldx$ we get $Q_\boldw(\boldx)  =  0  =  ( \sum_{\boldv \in C(\boldw, +)} (-1)^{\vert \boldw^{-1} \boldv \vert} P_\boldv  )(\boldx)$.
  In case $\boldw < \boldx$ we find
  \begin{equation}  \label{Eqn=ChooseSum}
 \left( \sum_{\boldv \in C(\boldw, + ) } (-1)^{\vert \boldw^{-1} \boldv \vert} P_{\boldv}    \right)(\boldx)
 =
  \sum_{\boldv \in C(\boldw, \boldx) } (-1)^{\vert \boldw^{-1} \boldv \vert},
  \end{equation}
 and this expression equals 0 by the binomial formula. Indeed, let $\Lambda_{\boldw, \boldx}$ be the maximal clique appearing at the start of $\boldw^{-1} \boldx$ (see Section \ref{Sect=CreationAnnihilation}). The number of words smaller than $\Lambda_{\boldw, \boldx}$ of length $l$ is $\vert \Lambda_{\boldw, \boldx} \vert$ choose $l$. So   \eqref{Eqn=ChooseSum} equals,
 \[
 \sum_{l = 0}^{\vert \Lambda_{\boldw, \boldx} \vert } \sum_{\boldv \in C(\boldw, \boldx), \vert \boldw^{-1} \boldv \vert = l} (-1)^{\vert \boldw^{-1} \boldv \vert }
 = \sum_{l = 0}^{\vert \Lambda_{\boldw, \boldx} \vert }  \left( \begin{array}{c} \vert \Lambda_{\boldw, \boldx} \vert \\ l \end{array} \right) (-1)^{\vert \boldw^{-1} \boldv \vert }  = 0.
 \]
  This concludes the lemma.
  \end{proof}

Now let $\mathcal{A}_q$ be the $\ast$-algebra generated by the operators $T_\boldw, \boldw \in W$. So $\cM_q$ is the $\sigma$-weak closure of $\mathcal{A}_q$. We define
\[
\Psi_{\leq n}: \mathcal{A}_q \rightarrow \cM_q: T_\boldw \mapsto \left\{
\begin{array}{ll}
T_\boldw & \vert \boldw \vert \leq n,\\
0 & {\rm otherwise}.
\end{array}
\right.
\]
We also set $\Psi_{n} = \Psi_{\leq n} - \Psi_{\leq (n-1)}$. The crucial part which we need to prove is that $\Psi_{\leq n}$ is completely bounded with a complete bound that can be upper estimated in $n$ polynomially.  In order to do so we first introduce 3   auxiliary maps.

\vspace{0.3cm}

\noindent {\bf Auxiliary map 1.} Recall that $\cM_1$ is just the group von Neumann algebra of the right-angled Coxeter group $W$. For $k \in \mathbb{N}$ define the multiplier  $\mathcal{A}_1 \rightarrow \mathcal{A}_1$,
\[
\rho_k(T_{\boldw}^{(1)})  = \delta(\vert \boldw \vert = k) T_{\boldw}^{(1)}.
\]
This map is completely bounded as the range is finite dimensional. We may extend $\rho_k$ to a $\sigma$-weakly continuous map $\cM_1 \rightarrow \cM_1$ (for convenience of the reader we provided details of this extension trick through double duality in Theorem \ref{Thm=CBAP}).
 By the Bozejko-Fendler Theorem \ref{Thm=BozejkoFendler} we may extend $\rho_k$ uniquely to a $\sigma$-weakly continuous $\ell^\infty(W)$-bimodule map $\cB(\ell^2(W)) \rightarrow \cB(\ell^2(W))$ with the same completely bounded norm. Using Lemma \ref{Lem=TExpansion} we see that
\[
\Psi_{\leq n} = \sum_{k=0}^n  \rho_k \circ  \Psi_{\leq n}.
\]
 We emphasize at this point that in our proofs we shall not need a growth estimate for $\Vert \rho_k \Vert_{\mathcal{CB}}$ in terms of $k$.  It is known however by \cite{Reckwerdt} that $\Vert \rho_k \Vert_{\CB}$ admits a polynomial bound in $k$. In the hyperbolic case this would already follow from \cite[Theorem 1 (2)]{OzawaWeak}.

  Only in the hyperbolic case it is known by \cite[Theorem 1 (2)]{OzawaWeak} that this map is completely bounded and moreover $\Vert \rho_k \Vert_{\CB} \leq C (k + 1)$ for some constant $C$ independent of $k$.

\vspace{0.3cm}

\noindent {\bf Auxiliary map 2.}  Let $\mathbb{T}$ be the unit circle in $\mathbb{C}$.  For $z \in \mathbb{T}$ we define a unitary map,
\[
A_z: \ell^2(W) \rightarrow \ell^2(W):  \delta_\boldw  \mapsto z ^{\vert \boldw \vert} \delta_\boldw.
\]
We set for $i \in \mathbb{Z}$,
\[
\Phi_i: \cB( \ell^2(W) ) \rightarrow \cB( \ell^2(W)):  x \mapsto   \int_{\mathbb{T}} z^{-i}   A_{z}^\ast x A_{z}   dz,
\]
where the measure is the normalized Lebesgue measure on $\mathbb{T}$. Intuitively $\Phi_i$ cuts out the operators that create $i$ more letters than it annihilates (where a negative creation is an annihilation).    Using Lemma \ref{Lem=TExpansion} we see that
\[
\Psi_{\leq n} =  \sum_{i=-n}^n \Phi_i  \circ \Psi_{\leq n}.
\]

\vspace{0.3cm}

\noindent {\bf Auxiliary map 3.} Assume that $\Gamma$ is finite. For $a \in \mathbb{N}$ we define  Stinespring dilations,
\begin{equation}\label{Eqn=Dilate}
U_{a}^\pm: \ell^2(W) \rightarrow \ell^2(W) \otimes \ell^2(W) \otimes \ell^2(W) \otimes \ell^2(B_f(W)),
\end{equation}
by mapping $\delta_\boldx$ to (see Section \ref{Sect=CreationAnnihilation} for notation),
\[
\sum_{\boldg \leq \boldx} \sum_{\Lambda \leq \boldg(2, \emptyset)} \beta^{\pm}_{\boldg, \boldx, \Lambda, a} \delta_{\boldg} \otimes\delta_{\boldg^{-1} \boldx}   \otimes     \delta_{\boldg(2, \Lambda)}   \otimes \widetilde{\xi}_{\Lambda}^{\pm}.
\]
Here,
\begin{equation}\label{Eqn=Beta}
\beta^{+}_{\boldg, \boldx, \Lambda, a} = \sum_{\boldv \in C(\boldg, \boldx)} (-1)^{\vert \boldg^{-1} \boldv \vert}  F_{\Lambda,a}(\boldv),
 \end{equation}
 where $F_{\Lambda,a}(\boldv) =1$ if
 \[
2 \vert \boldv(1, \Lambda) \vert + \vert \boldv(2, \Lambda) \vert \leq a,
 \]
 and else $F_{\Lambda,a}(\boldv) = 0$. We let $\beta^{-}_{\boldg, \boldx, \Lambda, a} = 1$ if $\beta^{+}_{\boldg, \boldx, \Lambda, a} \not = 0$ and  $\beta^{-}_{\boldg, \boldx, \Lambda, a} = 0$ otherwise.  Then set,
\begin{equation}\label{Eqn=SigmaAB}
 \sigma_{a,b}(x) = (U^-_{a})^\ast (x \otimes 1 \otimes 1 \otimes 1) U^+_{b}.
\end{equation}
The map $U_a^{\pm}$ is bounded with polynomial bound in $a$ by the following lemma.

\begin{lem}\label{Lem=UPol}
If $\Gamma$ is finite, the map $U_a^{\pm}$ is bounded. Moreover, there exists a polynomial $P$ such that $\Vert U_{a}^{\pm} \Vert \leq P(a)$.
\end{lem}
\begin{proof}
It follows by a comparison of the first two tensor legs in the definition of $U_a^{\pm}$ that the images of $\delta_{\boldx}, \boldx \in W$ are orthogonal vectors. Therefore it suffices to show that $\sup_{\boldx \in W} \Vert U_a^{\pm} \delta_{\boldx} \Vert$ is bounded polynomially. Now let $C = \sum_{\Lambda \in \Cliq(\Gamma)}  2^{\frac{1}{2} \vert \Lambda\vert}$. Then
\begin{equation}\label{Eqn=JosEnEdgar}
\begin{split}
 \Vert U_a \delta_{\boldx} \Vert  = &  \Vert
\sum_{\boldg \leq \boldx} \sum_{\Lambda \leq \boldg(2, \emptyset)} \beta^{\pm}_{\boldg, \boldx, \Lambda, a} \delta_{\boldg} \otimes\delta_{\boldg^{-1} \boldx}   \otimes     \delta_{\boldg(2, \Lambda)}   \otimes \widetilde{\xi}_{\Lambda}^{\pm}.
\Vert \\
\leq &
\sum_{\boldg \leq \boldx} \sum_{\Lambda \leq \boldg(2, \emptyset)}   \vert \beta^{\pm}_{\boldg, \boldx, \Lambda, a} \vert 2^{\frac{1}{2} \vert \Lambda \vert} \\
\leq & C \sum_{\boldg \leq \boldx}   \max_{\Lambda \in \Cliq(\Gamma)} \vert \beta^{\pm}_{\boldg, \boldx, \Lambda, a} \vert.
\end{split}
\end{equation}
In case
\begin{equation}\label{Eqn=AEstimateI}
a \leq 2 \vert \boldg(1, \Lambda) \vert + \vert \boldg(2, \Lambda) \vert,
\end{equation}
 then $\beta^{\pm}_{\boldg, \boldx, \Lambda, a} = 0$ by definition. \eqref{Eqn=AEstimateI} will certainly hold when $a \leq  \vert\boldg\vert$.
  Let $M$ be the maximum length of a clique in $\Cliq(\Gamma)$. Then if
  \begin{equation}\label{Eqn=AEstimateII}
  2 \vert \boldg(1, \Lambda) \vert + \vert \boldg(2, \Lambda) \vert \leq a -2M -1,
  \end{equation}
   we find that $\beta^{\pm}_{\boldg, \boldx, \Lambda, a} = 0$ by the binomial formula as for every $\boldv \in C(\boldg, \boldx)$ we have $F_{\Lambda, a}(\boldv) = 1$. \eqref{Eqn=AEstimateII} will certainly hold if $2 \vert \boldg \vert \leq a - 2M -1$.
       So \eqref{Eqn=JosEnEdgar} can be estimated by $C$ times the number of $\boldg \leq \boldx$ with
    \[
 \frac{1}{2}(a - 2 M - 1) \leq     \vert \boldg \vert \leq a.
    \]
But the number such $\boldg$'s grows polynomially in $a$, c.f. Lemma \ref{Lem=PolynomialBound}.
\end{proof}

\begin{lem}\label{Lem=Aux}
Let $\boldx \in W$. Let $\boldu', \boldu''\in W$ be such that $\vert \boldu'' \boldx \vert = \vert \boldx \vert - \vert \boldu'' \vert$, $\vert \boldu' \boldu'' \boldx \vert = \vert \boldx \vert - \vert \boldu''\vert + \vert \boldu' \vert$. Let $\boldv \in W$ be such that $(\boldu'')^{-1} \leq \boldv \leq \boldx$. Then,
\begin{equation}\label{Eqn=Sum}
\begin{split}
&   \sum_{\boldv \leq \boldg \leq \boldx} \beta^{+}_{\boldg, \boldx, \boldg(2, \emptyset) \backslash (\boldu'\boldu''\boldg)(2, \emptyset), a} \beta^{-}_{\boldu'\boldu''\boldg, \boldu'\boldu''\boldx, (\boldu'\boldu''\boldg)(2, \emptyset) \backslash \boldg(2, \emptyset), a - 2\vert \boldu' \vert + 2\vert \boldu'' \vert}\\
= &
\left\{
\begin{array}{ll}
 1 & \textrm{in case } 2 \vert \boldv(1,  \boldv(2, \emptyset) \backslash (\boldu' \boldu''\boldv)(2, \emptyset)  ) \vert + \vert \boldv(2,  \boldv(2, \emptyset) \backslash (\boldu' \boldu''\boldv)(2, \emptyset)  ) \vert \leq a,\\
 0 & \textrm{otherwise.}
\end{array}
\right.
\end{split}
\end{equation}
\end{lem}
\begin{proof}
By Equations \eqref{Eqn=LaChouffe3} and \eqref{Eqn=LaChouffe4} for $\boldv \leq \boldg$ we get,
\[
\begin{split}
   & \beta^{+}_{\boldg, \boldx, \boldg(2, \emptyset) \backslash (\boldu'\boldu''\boldg)(2, \emptyset), a}\\
=  & \sum_{\boldw \in C(\boldg, \boldx)} (-1)^{\vert \boldg^{-1} \boldv \vert}  F_{\boldg(2, \emptyset) \backslash (\boldu'\boldu''\boldg)(2, \emptyset), a}(\boldw) \\
=  & \sum_{\boldw \in C(\boldu'\boldu''\boldg, \boldu'\boldu''\boldx) }  (-1)^{\vert \boldg^{-1} \boldv \vert} F_{ (\boldu'\boldu''\boldg)(2, \emptyset) \backslash \boldg(2, \emptyset), a - 2\vert \boldu'' \vert + 2\vert \boldu'\vert}(\boldw) \\
=  & \beta^{+}_{\boldu'\boldu''\boldg, \boldu'\boldu''\boldx, (\boldu'\boldu''\boldg)(2, \emptyset) \backslash \boldg(2, \emptyset), a - 2\vert \boldu' \vert + 2\vert \boldu'' \vert}.
\end{split}
\]
Therefore also,
\[
\beta^{-}_{\boldg, \boldx, \boldg(2, \emptyset) \backslash (\boldu'\boldu''\boldg)(2, \emptyset), a}
= \beta^{-}_{\boldu'\boldu''\boldg, \boldu'\boldu''\boldx, (\boldu'\boldu''\boldg)(2, \emptyset) \backslash \boldg(2, \emptyset), a - 2\vert \boldu' \vert + 2\vert \boldu'' \vert}.
\]
We therefore have that the left hand side of \eqref{Eqn=Sum} equals,
\[
 \sum_{\boldv \leq \boldg \leq \boldx} \beta^{+}_{\boldg, \boldx, \boldg(2, \emptyset) \backslash (\boldu'\boldu''\boldg)(2, \emptyset), a} \beta^{-}_{\boldg, \boldx, \boldg(2, \emptyset) \backslash (\boldu'\boldu''\boldg)(2, \emptyset), a} =
 \sum_{\boldv \leq \boldg \leq \boldx} \beta^{+}_{\boldg, \boldx, \boldg(2, \emptyset) \backslash (\boldu'\boldu''\boldg)(2, \emptyset), a}
\]
To compute this sum, recall that $\mathcal{R}$ was defined in \eqref{Eqn=Rspan}, and define the mapping
\[
\kappa_a: \mathcal{R} \rightarrow \mathcal{R}: P_{\boldw}\mapsto F_{\boldw(2, \emptyset) \backslash (\boldu'\boldu''\boldw)(2, \emptyset), a}(\boldw) P_{\boldw}.
\]
Then, using Lemma \ref{Lem=Aux2},  the definition of $\kappa_a$, Lemma \ref{Lem=Tediously}, and the definition \eqref{Eqn=Beta},
\[
\begin{split}
\kappa_a(Q_\boldg)(\boldx) = &  \kappa_a(\sum_{\boldw \in C(\boldg, \boldx)} (-1)^{\vert \boldg^{-1} \boldw \vert} P_{\boldw} )(\boldx)  \\
= &  \sum_{\boldw \in C(\boldg, \boldx)} (-1)^{\vert \boldg^{-1} \boldw \vert} F_{\boldw(2, \emptyset) \backslash (\boldu'\boldu''\boldw)(2, \emptyset), a}(\boldw)  \\
= & \beta^{+}_{\boldg, \boldx, \boldg(2, \emptyset) \backslash (\boldu'\boldu''\boldg)(2, \emptyset), a}.
\end{split}
\]
 As $\sum_{\boldv \leq \boldg \leq \boldx} Q_\boldg$ can be written as $P_\boldv$ plus projections in $\mathcal{R}$ that are not supported at $\boldx$ we see therefore that,
\[
\begin{split}
 \sum_{\boldv \leq \boldg \leq \boldx} \beta^{+}_{\boldg, \boldx, \boldg(2, \emptyset) \backslash (\boldu'\boldu''\boldg)(2, \emptyset), a}  = &  \sum_{\boldv \leq \boldg \leq \boldx}\kappa_a(Q_\boldg)(\boldx)
 =   \kappa_a(P_\boldv)(\boldx).
\end{split}
\]
This expression equals 1 if $F_{\boldv(2, \emptyset) \backslash (\boldu'\boldu''\boldv)(2, \emptyset), a}(\boldv) = 1$ and 0 otherwise which corresponds exactly to the statement of the lemma.

\end{proof}

\begin{lem} \label{Lem=FirstTediousLemma}
Assume that $\Gamma$ is finite so that \eqref{Eqn=SigmaAB} is defined boundedly. We have for $n \in \mathbb{N}$ that $\Psi_{\leq n} =     \sum_{i = -n}^n   \sigma_{   n-i,   n+i  } \circ \Phi_{i}  \circ \Psi_{\leq n}$.
\end{lem}
\begin{proof}
Let $T_\boldw \in \cM_q$ with $\vert \boldw \vert \leq n$. We need to show that,
\[
T_\boldw  =  \sum_{k=0}^{n} \sum_{i = -n}^n   \sigma_{   n-i,   n+i  } \circ \Phi_{i} \circ \rho_k(T_\boldw).
\]
We split by Lemma \ref{Lem=TExpansion},
\[
T_\boldw =   \sum_{(\boldw', \Gamma_0, \boldw'') \in A_\boldw} T_{\boldw'}^{(1)} P_{V\Gamma_0} T_{\boldw''}^{(1)},
\]
and show that $ \sum_{k=0}^{n} \sum_{i = -n}^n   \sigma_{   n-i,   n+i  } \circ \Phi_{i} \circ \rho_k $ applied to each of these summands acts as the identity.
Let us consider a summand $T_{\boldw'}^{(1)} P_{V\Gamma_0} T_{\boldw''}^{(1)}$ with $(\boldw', \Gamma_0, \boldw'')\in A_{\boldw}$. Let $\boldu, \boldu', \boldu''$ be as in Lemma \ref{Lem=BreakDown} so that $T_{\boldw'}^{(1)} P_{V\Gamma_0} T_{\boldw''}^{(1)} = T_{\boldu'}^{(1)} P_{\boldu V\Gamma_0} T_{\boldu''}^{(1)}$. We have
\[
\rho_k(T_{\boldu'}^{(1)} P_{\boldu V\Gamma_0} T_{\boldu''}^{(1)}) = \left\{
\begin{array}{ll}
T_{\boldu'}^{(1)} P_{\boldu V\Gamma_0} T_{\boldu''}^{(1)} & \textrm{if } k = \vert \boldu'\vert + \vert \boldu''\vert, \\
0 & \textrm{otherwise.}
\end{array}
\right.
\]
So the only non-zero summand is $k = \vert \boldu' \vert + \vert \boldu'' \vert$ so that it remains to show that for $\boldx, \boldy \in W$,
\begin{equation}\label{Eqn=MoonenIsAnAsshol}
\langle
\sum_{i = -n}^n   \sigma_{   n-i,   n+i  } \circ \Phi_{i}(    T_{\boldu'}^{(1)} P_{\boldu V\Gamma_0} T_{\boldu''}^{(1)}   ) \delta_{\boldx}, \delta_{\boldy}
\rangle = \langle
    T_{\boldu'}^{(1)} P_{\boldu V\Gamma_0} T_{\boldu''}^{(1)}   \delta_{\boldx}, \delta_{\boldy}
\rangle.
\end{equation}
If the right hand side is non-zero then we must have $\boldy = \boldu' \boldu'' \boldx$. Furthermore, recall that there is a choice for $\boldu', \boldu''$ and we may choose them (depending on $\boldx$) such that  $\vert \boldu'' \boldx \vert = \vert \boldx \vert - \vert \boldu'' \vert$ and $\vert \boldu' \boldu'' \boldx \vert = \vert \boldx \vert - \vert \boldu'' \vert + \vert \boldu'\vert$. After making this choice the right hand side is non-zero in case $(\boldu'')^{-1} \boldu V\Gamma_0  \leq \boldx$, in which case the expression equals 1.

Now consider the left hand side of \eqref{Eqn=MoonenIsAnAsshol},
 \begin{equation}\label{Eqn=BigClaim}
 \begin{split}
 &  \langle    (   \Phi_i(T_{\boldu'}^{(1)} P_{\boldu V\Gamma_0} T_{\boldu''} ^{(1)}) \otimes 1  \otimes 1 \otimes 1 ) U^+_{ n-i }  \delta_{\boldx}, U^-_{  n+i}  \delta_{\boldy}   \rangle \\
   = & \langle \sum_{\boldg \leq \boldx} \sum_{\Lambda \leq \boldg(2, \emptyset)} \beta^+_{\boldg, \boldx, \Lambda, n-i } \Phi_i( T_{\boldu'}^{(1)} P_{\boldu V\Gamma_0} T_{\boldu''} ^{(1)} ) \delta_\boldg \otimes \delta_{\boldg^{-1} \boldx} \otimes  \delta_{\boldg(2, \Lambda)} \otimes \widetilde{\xi}^+_{\Lambda}, \\
   & \qquad \sum_{\boldh \leq \boldy} \sum_{\Lambda' \leq \boldh(2, \emptyset)} \beta^-_{\boldh, \boldy, \Lambda', n+i } \delta_{\boldh}  \otimes \delta_{\boldh^{-1} \boldy} \otimes \delta_{\boldh(2, \Lambda')} \otimes \widetilde{\xi}_{\Lambda}^- \rangle.
   \end{split}
   \end{equation}
Comparing the first two tensor legs of this equation we derive the following. The only summands that are non-zero are the ones where $\boldu' \boldu'' \boldg = \boldh$ and at the same time  $\boldg^{-1} \boldx = \boldh^{-1} \boldy$. In particular we must have $\boldy = \boldu' \boldu'' \boldx$ and there is a choice for $\boldu', \boldu''$ (same choice as above) such that in fact $\vert \boldu'' \boldx \vert = \vert \boldx \vert - \vert \boldu'' \vert$ and $\vert \boldu' \boldu'' \boldx \vert = \vert \boldx \vert - \vert \boldu'' \vert + \vert \boldu' \vert$. We also see that we must have $(\boldu'')^{-1} \boldu V\Gamma_0 \leq \boldx$ for this expression to be non-zero. Taking into account $\Phi_i$ we see that \eqref{Eqn=BigClaim} is non-zero only if $i = \vert \boldu ''\vert - \vert \boldu' \vert$.

Next we note that  by comparing the last two tensor legs,  if a summand in \eqref{Eqn=BigClaim} is non-zero then we have $\boldg(2, \Lambda) = \boldh(2, \Lambda')$ and $\Lambda \cap \Lambda' = \emptyset$. Recall that $\boldh = \boldu' \boldu'' \boldg$. But then $\Lambda$ must equal the letters in $\boldg(2, \emptyset)$ that are not any more in $(\boldu' \boldu'' \boldg)(2, \emptyset)$ and $\Lambda'$ must equal the letters in $(\boldu' \boldu'' \boldg)(2, \emptyset)$ that are not anymore in $\boldg(2, \emptyset)$. This precisely means that $\Lambda = \boldg(2, \emptyset) \backslash (\boldu' \boldu'' \boldg)(2, \emptyset)$ and $\Lambda' = (\boldu' \boldu'' \boldg)\backslash \boldg(2, \emptyset)$.

In all, we find that
 \[
 \begin{split}
 \eqref {Eqn=BigClaim} = &   \langle    (T_{\boldu'}^{(1)} P_{\boldu V\Gamma_0} T_{\boldu''} ^{(1)} \otimes 1  \otimes 1 \otimes 1 ) U^+_{ n-i }  \delta_{\boldx}, U^-_{  n+i}  \delta_{\boldy}   \rangle \\
   = & \langle \sum_{\boldg \leq \boldx} \sum_{\Lambda \leq \boldg(2, \emptyset)} \beta^+_{\boldg, \boldx, \Lambda, n-i } T_{\boldu'}^{(1)} P_{\boldu V\Gamma_0} T_{\boldu''} ^{(1)} \delta_\boldg \otimes \delta_{\boldg^{-1} \boldx} \otimes  \delta_{\boldg(2, \Lambda)} \otimes \widetilde{\xi}^+_{\Lambda}, \\
   & \qquad \sum_{\boldh \leq \boldy} \sum_{\Lambda' \leq \boldh(2, \emptyset)} \beta^-_{\boldh, \boldy, \Lambda', n+i } \delta_{\boldh^{-1} \boldx} \otimes \delta_{\boldh} \otimes \delta_{\boldh(2, \Lambda')} \otimes \widetilde{\xi}_{\Lambda}^- \rangle \\
   = & \sum_{(\boldu'')^{-1} \boldu V\Gamma_0 \leq \boldg \leq \boldx} \beta^+_{\boldg, \boldx, \boldg(2, \emptyset) \backslash (\boldu' \boldu'' \boldg)(2, \emptyset),   n-i }  \beta^-_{\boldu' \boldu'' \boldg, \boldu' \boldu'' \boldx, (\boldu' \boldu'' \boldg)(2, \emptyset) \backslash  \boldg(2, \emptyset),   n+i }.\\
   \end{split}
   \]

 We claim that this expression is 1 by verifying Lemma \ref{Lem=Aux}. Indeed
set  $\boldw := (\boldu'')^{-1} \boldu V\Gamma_0$. First suppose that $\boldu$ is the empty word. Then
\[
\boldw(2, \boldw(2, \emptyset) \backslash (\boldu' \boldu'' \boldw)(2, \emptyset)  ) = V\Gamma_0
\]
and so
\[
\boldw(1, \boldw(2, \emptyset) \backslash (\boldu' \boldu'' \boldw)(2, \emptyset)  ) = (\boldu'')^{-1}.
 \]
 If $\boldu$ is not the empty word, then let $s\in W$ be a final letter of $\boldu$ (i.e. $\vert \boldu s \vert = \vert \boldu \vert - 1$). Then $s$ cannot commute with $V\Gamma_0$ as this would violate the equation $T_{\boldu'}^{(1)} P_{\boldu V\Gamma_0} T_{\boldu''} ^{(1)} = T_{\boldw'}^{(1)} P_{  V\Gamma_0} T_{\boldw''} ^{(1)}$. Therefore again,
 \[
 \boldw(2, \boldw(2, \emptyset) \backslash (\boldu' \boldu'' \boldw)(2, \emptyset)  ) = \boldw(2, \emptyset) = V\Gamma_0
  \]
  and so
  \[
  \boldw(1, \boldw(2, \emptyset) \backslash (\boldu' \boldu'' \boldw)(2, \emptyset)  ) = (\boldu'')^{-1} \boldu.
   \]
   Further our constructions give that $\vert \boldu''\vert = \frac{k-i}{2}$ and $2 \vert \boldu \vert + \vert V\Gamma_0 \vert = \vert \boldw \vert - \vert \boldu' \vert - \vert \boldu''\vert = \vert \boldw \vert - k$. So we have,
 \begin{equation}\label{Eqn=House}
 \begin{split}
&  2 \vert \boldw(1, \boldw(2, \emptyset) \backslash (\boldu' \boldu'' \boldw)(2, \emptyset)  ) \vert +
 \vert \boldw(2, \boldw(2, \emptyset) \backslash (\boldu' \boldu'' \boldw)(2, \emptyset)  ) \vert \\
 =& 2 \vert (\boldu'')^{-1}   \vert + 2 \vert \boldu \vert + \vert V\Gamma_0 \vert
  =  2 \frac{k-i}{2} + (\vert \boldw \vert - k) \\
  = &
  \vert \boldw \vert - i  \leq n-i,
 \end{split}
 \end{equation}
 so that  by Lemma \ref{Lem=Aux} we see that \eqref{Eqn=BigClaim} is 1.
So we conclude that \eqref{Eqn=MoonenIsAnAsshol} holds.

\end{proof}

\begin{lem} \label{Lem=SecondTediousLemma}
Assume that $\Gamma$ is finite so that \eqref{Eqn=SigmaAB} is defined boundedly. We have for $n \in \mathbb{N},  -n \leq i \leq n$:
\[
    \sigma_{   n-i ,   n+i} \circ \Phi_{i} \circ  \Psi_{\leq n} =      \sigma_{   n-i ,   n+i } \circ \Phi_{i}.
\]
\end{lem}
\begin{proof}
The proof pretty much parallels the proof of Lemma \ref{Lem=FirstTediousLemma}.
We need to show that the right hand side applied to $T_\boldw$ with $\vert \boldw \vert > n$ equals 0. Therefore we may look at the summands $T_{\boldw'}^{(1)} P_{V\Gamma_0}T_{\boldw''}^{(1)}$ with $(\boldw', \Gamma_0, \boldw'') \in A_\boldw$ which can be further decomposed as $T_{\boldu'}^{(1)} P_{\boldu V\Gamma_0} T_{\boldu''}^{(1)}$ with  $\boldu, \boldu', \boldu''$ as in Lemma \ref{Lem=BreakDown}.   It suffices then to show that for all choices of $k$ the following expression is 0:
\begin{equation}\label{Eqn=Koffie}
\langle   \sigma_{   n-i ,   n+i } \circ \Phi_{i} \circ \rho_k  (T_{\boldu'}^{(1)} P_{\boldu V\Gamma_0} T_{\boldu''}^{(1)}) \delta_{\boldx}, \delta_{\boldy} \rangle.
\end{equation}
Firstly,     this expression is 0 in case $\vert \boldu' \vert + \vert \boldu''\vert \not = k$. So assume  $\vert \boldu' \vert + \vert \boldu''\vert   = k$. Then,
\[
 \eqref{Eqn=Koffie} = \langle   \sigma_{   n-i ,   n+i } \circ \Phi_{i}   (T_{\boldu'}^{(1)} P_{\boldu V\Gamma_0} T_{\boldu''}^{(1)}) \delta_\boldx, \delta_\boldy \rangle.
 \]
  As in the proof of Lemma \ref{Lem=FirstTediousLemma} the expression \eqref{Eqn=Koffie} equals 0 unless $\boldu' \boldu'' \boldx = \boldy$ and $(\boldu'')^{-1} \boldu V\Gamma_0 \leq \boldx$ with $\boldu'', \boldu'$ chosen in such a way that $\vert \boldu'' \boldx \vert = \vert \boldx \vert - \vert \boldu''\vert$ and $\vert \boldu' \boldu'' \boldx \vert = \vert \boldx \vert - \vert \boldu''\vert + \vert \boldu' \vert$. In that case $i = \vert \boldu' \vert - \vert \boldu''\vert$. As in \eqref{Eqn=BigClaim},
 \begin{equation}\label{Eqn=NogEen}
 \begin{split}
   \eqref{Eqn=Koffie} =  &  \langle    (T_{\boldu'}^{(1)} P_{\boldu V\Gamma_0} T_{\boldu''} ^{(1)}  \otimes 1 \otimes   1 \otimes 1 ) U^+_{ n-i }  \delta_{\boldx}, U^-_{  n+i }  \delta_{\boldy}   \rangle \\
   = &      \sum_{(\boldu'')^{-1} \boldu V\Gamma_0 \leq \boldg \leq \boldx} \beta_{\boldg, \boldx, \boldg(2, \emptyset) \backslash (\boldu' \boldu'' \boldg)(2, \emptyset),   n-i }  \beta_{\boldu' \boldu'' \boldg, \boldu' \boldu'' \boldx, (\boldu' \boldu'' \boldg)(2, \emptyset) \backslash  \boldg(2, \emptyset),   n+i }.
 \end{split}
 \end{equation}
 As for  $\boldw := (\boldu'')^{-1} \boldu V\Gamma_0$ we have again by the same reasoning as in/before   \eqref{Eqn=House} that,
 \[
 \begin{split}
&  2 \vert \boldw(1, \boldw(2, \emptyset) \backslash (\boldu' \boldu'' \boldw)(2, \emptyset)  ) \vert +
 \vert \boldw(2, \boldw(2, \emptyset) \backslash (\boldu' \boldu'' \boldw)(2, \emptyset)  ) \vert  =
  \vert \boldw \vert - i  > n-i.
 \end{split}
 \]
 The expression \eqref{Eqn=NogEen} is zero by Lemma \ref{Lem=Aux}.
\end{proof}
\begin{prop}\label{Prop=CutDown}
We have $\Vert \Psi_{\leq n} \Vert_{\CB} \leq P(n)$ for some polynomial $P$.
\end{prop}
\begin{proof}
 By Lemmas \ref{Lem=FirstTediousLemma} and \ref{Lem=SecondTediousLemma} we have,
 \[
 \begin{split}
 \Psi_{\leq n}  = &    \sum_{i = -n}^n    \sigma_{   n-i ,  n+i  } \circ \Phi_{i}  \circ \Psi_{\leq n}\\
  = &     \sum_{i = -n}^n    \sigma_{   n-i,   n+i } \circ \Phi_{i},
 \end{split}
 \]
 and the right hand side is completely bounded with polynomial bound in $n$; indeed the bound of $\sigma_{   n-i ,  n+i  }$ is polynomial in $n$ by its very definition and Lemma  \ref{Lem=UPol}.
\end{proof}

\begin{dfn}\label{Dfn=CBAP}
A von Neumann algebra $\cM$ has the weak-$\ast$ completely bounded approximation property (wk-$\ast$ CBAP) if there exists a net of normal finite rank maps $\Phi_i: \cM \rightarrow \cM$ such that $\Phi_i(x) \rightarrow x$ in the $\sigma$-weak topology and moreover $\sup_{i} \Vert \Phi_i \Vert_{\CB} < \infty$. If the maps $\Phi_i$ can be chosen so that $\limsup_{i} \Vert \Phi_i \Vert_{\CB} \leq 1$ then $\cM$ is said to have the weak-$\ast$ completely contractive approximation property (wk-$\ast$ CCAP).
\end{dfn}

\begin{thm}\label{Thm=CBAP}
Let $(W, S)$ be a right angled Coxeter system and let $q> 0$. The Hecke von Neumann algebra $\cM_q$ has the wk-$\ast$ CCAP.
\end{thm}
\begin{proof}
By an inductive limit argument and Corollary \ref{Cor=Expected} we may assume that $\Gamma$ is finite.
The proof  goes back to Haagerup \cite{HaagerupExample}. Consider the completely bounded map $\Psi_{\leq n} \circ \Phi_r: \mathcal{A}_q \rightarrow \cM_q$. Clearly as $n \rightarrow \infty$ and $r \nearrow 1$ this map converges to the identity in the point $\sigma$-weak topology. Let $\epsilon > 0$. We have,
\[
\Vert \Psi_{\leq n} \circ \Phi_r \Vert_{\CB} \leq \Vert (\Psi_{\leq n} - {\rm Id}) \circ \Phi_r  \Vert_{\CB} + \Vert \Phi_r \Vert_{\CB}
\leq  \left( \sum_{i=n+1}^{\infty} r^n \Vert \Psi_n \Vert_{\CB} \right) + \Vert \Phi_r \Vert_{\CB},
\]
which shows using Proposition \ref{Prop=CutDown} and Proposition \ref{Prop=Radial} that we may  let $r \nearrow 1$ and then choose $n := n_r$ converging to $\infty$ such that $\Vert \Psi_{\leq n_r} \circ \Phi_{r} \Vert_{\CB} \leq 1 + \epsilon$ for some constant.

The map $\Phi_r$ is normal. Also $\Psi_{\leq n}$ is normal by a standard argument: indeed using duality and Kaplansky's density theorem one sees that   $\Psi_n$ maps $L^1(\cM_q) \rightarrow L^1(\cM_q)$ boundedly. Then taking the dual of this map yields that $\Psi_n: \cM_q \rightarrow \cM_q$ is a normal map.  We may extend  $\Psi_{\leq n} \circ \Phi_r$  to a normal map $\cM_q \rightarrow \cM_q$. Then using a standard approximation argument yields the result.
\end{proof}

\begin{rmk}
In case our right-angled Coxeter group is free (i.e. $m(s,t) = \infty$ for all $s \not = t$) it is possible to adapt the arguments of \cite{RicardXu} in order to obtain word length cut downs with polynomial bound. This argument -- purely based on book keeping of creations/annihilations -- seems unrepairable in the general case.   In case $q = 1$ for a general right-angled Coxeter group  word length cut-downs were obtained in \cite{Reckwerdt} by using actions on CAT(0)-spaces. The connection with the general Hecke case is unclear.
\end{rmk}

\section{Strong solidity in the hyperbolic case}\label{Sect=StrongSolidity}

We prove that in the factorial case (see Theorem \ref{Thm=Factor})   $\cM_q$   is a strongly solid von Neumann algebra in case the Coxeter group is hyperbolic.

\subsection{Preliminaries on strongly solid algebras}
 The {\it normalizer} of a von Neumann subalgebra  $\mathcal{P}$ of $\mathcal{M}$ is defined as $\{ u \in \mathcal{U}(\cM) \mid  u \mathcal{P} u^\ast = \mathcal{P} \}$. We define $\Nor_{\cP}(\cM)$ as the von Neumann algebra generated by the normalizer of $\cP$ in $\cM$.  A von Neumann algebra is called {\it diffuse} if it does not contain minimal projections.

 \begin{dfn}\label{Dfn=StrongSolidity}
 A  finite von Neumann algebra $\cM$ is {\it strongly solid} if for any diffuse injective von Neumann subalgebra $\mathcal{P} \subseteq \cM$ the von Neumann algebra  $\Nor_{\mathcal{M}}(\mathcal{P})$ is  again injective.
 \end{dfn}

In \cite{OzawaPopaII} Ozawa and Popa proved that free group factors are strongly solid and consequently they could prove that these are II$_1$ factors that have no Cartan subalgebras (as was proved by Voiculescu \cite{Voiculescu} earlier on by a completely different method). A general source for strongly solid von Neumann algebras are group von Neumann algebras of groups that have the weak-$\ast$ completely bounded approximation property and are bi-exact (see \cite{ChifanSinclair}, \cite{ChifanSinclairUdrea}, \cite{PopaVaesCrelle}; we also refer to these sources for the definition of bi-exactness).
  The following definition and subsequent theorem were then  introduced and proved in \cite{IsonoExample}. For standard forms of von Neumann algebras we refer to \cite{Takesaki2}.

 \begin{dfn}\label{Dfn=AO}
 Let $\mathcal{M} \subseteq \cB(\cH)$ be a von Neumann algebra represented on the standard Hilbert space $\cH$ with modular conjugation $J$. We say that $\cM$ satisfies condition (AO)$^+$ if there exists a unital C$^\ast$-subalgebra $\bA \subseteq \cM$ that is $\sigma$-weakly dense in $\cM$ and which satisfies the following two conditions:
 \begin{enumerate}
 \item $\bA$ is locally reflexive.
 \item\label{Item=AO2} There exists a ucp map $\theta: \bA \minotimes J \bA J \rightarrow \cB(\cH)$ such that $\theta(a \otimes b) -ab$ is a compact operator on $\cH$.
 \end{enumerate}
 \end{dfn}

\begin{thm}[\cite{IsonoExample}]\label{Thm=AOStrongSolidity}
Let $\cM$ be a II$_1$-factor with separable predual. Suppose that $\cM$ satisfies condition (AO)$^+$ and has the weak-$\ast$ completely bounded approximation property. Then $\cM$ is strongly solid.
\end{thm}

A maximal abelian von Neumann subalgebra $\mathcal{P} \subseteq \cM$ of a II$_1$ factor $\cM$ is called a {\it Cartan subalgebra} if $\Nor_{\mathcal{M}}(\mathcal{P}) = \mathcal{M}$. It is then obvious that if $\cM$ is a non-injective strongly solid II$_1$-factor, then $\cM$ cannot contain a Cartan subalgebra. Therefore we will now prove that the Hecke von Neumann algebra $\cM_q$ in the factorial, hyperbolic case satisfies condition (AO)$^+$.

\subsection{Crossed products}
Let $\bA$ be a C$^\ast$-algebra that is represented on a Hilbert space $\cH$. Let $\alpha: \sG \curvearrowright \bA$ be a continuous action of a discrete  group $\sG$ on $\bA$. The reduced crossed product $\bA \rtimes_r \sG$ is the C$^\ast$-algebra of operators acting on $\cH \otimes \ell^2(\sG)$ generated by operators
\begin{equation}\label{Eqn=Generators}
u_g := \sum_{h \in \sG} 1 \otimes e_{gh, h}, \quad g \in \sG, \qquad \textrm{ and } \qquad \pi(x) := \sum_{h \in \sG}   h^{-1} \cdot x \otimes e_{h,h}, \quad x \in \bA.
\end{equation}
Here the convergence of the sums should be understood in the strong topology. There is also a universal crossed product $\bA \rtimes_u \sG$ for which we refer to \cite{BrownOzawa} (in the case we need it, it turns out to equal the reduced crossed product).

\subsection{Gromov boundary and condition (AO)$^+$}

Let again $(W,S)$ be a Coxeter system which we assume to be hyperbolic (see \cite[Section 5.3]{BrownOzawa}). Let $\Lambda$ be the associated Cayley tree. A geodesic ray starting at a point $\boldw \in \Lambda$ is a sequence $(\boldw, \boldw v_1, \boldw v_1 v_2, \ldots)$ such that $\vert \boldw v_1 \ldots v_n \vert = \vert \boldw \vert + n$. We typically write $\omega = (\omega(0), \omega(1), \ldots)$ for a geodesic ray. Let $\partial W$ be the Gromov boundary of $W$ which is the collection of all geodesic rays starting at the identitiy of $W$ modulo the equivalence $\omega_1 \simeq \omega_2$ iff $\lim_{x,y \rightarrow \infty} dist(\omega_1(x), \omega_2(y)) = 0$. $W \cup \partial W$ may be topologized  as in \cite[Section 5.3]{BrownOzawa}.

Let $W \curvearrowright W$ be the action by means of left translation. The action extends continuously to $W \cup \partial W$ and then restricts to an action $W \curvearrowright \partial W$. We may pull back this action to obtain $W \curvearrowright C(\partial W)$.
As in this section we assumed that $W$ is a hyperbolic group the  action $W \curvearrowright \partial W$ is well-known to be amenable \cite{BrownOzawa} which implies that $C(\partial W) \rtimes_u W = C(\partial W) \rtimes_r W$ and furthermore this crossed product is a nuclear C$^\ast$-algebra. Let $f \in C(\partial W)$, let $\widetilde{f}_1, \widetilde{f}_2 \in C(W \cup \partial W)$ be two continuous extensions of $f$ and let $f_1$ and $f_2$ be their respective restrictions to $W$. Then $f_1 - f_2 \in C_0(W)$. That is, the multiplication map $f_1 - f_2$ acting on $\ell^2(W)$ determines a compact operator. So the assignment $f \mapsto f_1$ is a well-defined $\ast$-homomorphism $C(\partial W) \rightarrow \cB(\ell^2(W)) / \cK$ where $\cK$ are the compact operators on $\ell^2(W)$. It is easy to check that this map is $W$-equivariant and thus we obtain a $\ast$-homomorpism:
\begin{equation}\label{Eqn=PiOne}
\pi_1: C(\partial W) \rtimes_u W \rightarrow \cB(\ell^2(W)) / \cK.
\end{equation}

Let again $W \curvearrowright W$ be the action by means of left translation which may be pulled back to obtain an action $W \curvearrowright \ell^\infty(W)$. Let
\[
\rho: \ell^\infty(W) \rtimes_r W \rightarrow \cB(\ell^2(W))
 \]
 be the $\sigma$-weakly continuous $\ast$-isomorphism determined by $\rho: u_\boldw \mapsto T_\boldw^{(1)}$ and $\rho: \pi(x) \mapsto x$ (see \cite[Theorem 5.3]{VaesJFA}). In fact $\rho$ is an injective map (this follows immediately from \cite[Theorem 2.1]{DeCommer} as the operator $G$ in this theorem equals the multiplicative unitary/structure operator \cite[p. 68]{Takesaki2}).
  Let $C_\infty(W)$ be the C$^\ast$-algebra generated by the projections $P_{\boldw}, \boldw \in W$. Take $f \in C_\infty(W)$ and let $\widetilde{f}$ be the continuous extension of $f$ to $W \cup \partial W$. The map $f \mapsto \widetilde{f}\vert_{\partial W}$ determines a $\ast$-homomorphism $\sigma: C_\infty(W) \rightarrow C(\partial W)$ that is $W$-equivariant. Therefore it extends to the crossed product map
  \[
  \sigma \rtimes_r \Id: C_\infty(W) \rtimes_r W \rightarrow C(\partial W) \rtimes_r W.
  \]

\begin{thm}\label{Thm=APplus}
Let $(W,S)$ be a right-angled hyperbolic Coxeter group and let $q \in [\rho, \rho^{-1}]$, see Theorem \ref{Thm=Factor}. The von Neumann algebra $\cM_q$ satisfies condition (AO)$^+$.
\end{thm}
\begin{proof}
We let $\bA_q$ be the unital C$^\ast$-subalgebra of $\cM_q$ generated by operators $T_\boldw, \boldw \in W$. It is easy to see that $\bA_q$ is preserved by the multipliers that we constructed in order to prove that $\cM_q$ had the wk-$\ast$ CBAP, see Section \ref{Sect=Approximation} (indeed these were compositions of radial multipliers -- see Proposition \ref{Prop=Radial} -- and word length projections -- see Proposition \ref{Prop=CutDown}).  Therefore $\bA_q$ has the CBAP, hence by the remarks before \cite[Theorem 2.2]{HaagerupKraus} it is exact. Therefore $\bA_q$ is locally reflexive  \cite{BrownOzawa}, \cite[Chapter 18]{Pisier}.

It remains to prove condition \eqref{Item=AO2} of Definition \ref{Dfn=AO}. By Lemma \ref{Lem=TExpansion} we see that $\bA_q$ is contained in the C$^\ast$-subalgebra of $\cB(\ell^2(W))$ generated by the operators $P_\boldw, T_\boldw^{(1)}$ with  $\boldw \in W$. So $\rho^{-1}(\bA_q)$ is contained in $C_{\infty}(W) \rtimes_r W$ and therefore we may set
\[
\gamma: \bA_q \rightarrow C(\partial W) \rtimes_r W \qquad \textrm{ as } \qquad  \gamma = (\sigma \rtimes_r \Id) \circ \rho^{-1}.
\]
 The mapping $\pi_2: J \bA_q J \rightarrow \cB(\ell^2(W)) \slash \mathcal{K} : b \mapsto b$ is a $\ast$-homomorphism and its image commutes with the image of $\pi_1$ of \eqref{Eqn=PiOne} (as was argued in \cite[Lemma 6.2.8]{GuentnerHigson}). By definition of the maximal tensor product there exists a $\ast$-homomorphism:
\[
( \pi_1 \otimes \pi_2): (C(\partial W) \rtimes_u W) \otimes_{{\rm max}} J \bA_q J \rightarrow \cB(\ell^2(W)) \slash \cK: a \otimes JbJ \mapsto \pi_1(a) JbJ.
\]
We may now consider the following composition of $\ast$-homomorphisms:
\begin{equation}\label{Eqn=Diagram1}
\xymatrix{
\bA_q \minotimes J \bA_q J \ar@{->}[r]^{\!\!\!\!\!\!\!\!\!\!\!\!\!\!\!\!\! \gamma \otimes \id }  &   (C(\partial W) \rtimes_r W) \otimes_{{\rm min}} J \bA_q J  \ar@{->}[d]^{ \simeq }      \\
\cB(\ell^2(W)) \slash \cK &   (C(\partial W) \rtimes_u W) \otimes_{{\rm max}} J \bA_q J   \ar@{^{(}->}[l]^{\!\!\!\!\!\!\!\!\!\!\!\!\!\!\!\!\!\!\!\!\!\!\!\! \pi_1 \otimes \pi_2 }.
}
\end{equation}
By construction this map is given by:
\begin{equation}\label{Eqn=ProductModCompacts}
a \otimes J b J \mapsto a JbJ + \cK, \qquad \textrm{where} \quad a,b \in \bA_q.
\end{equation}
The map $\pi_1$ is nuclear because we already observed that $C(\partial W) \rtimes_u W$ is nuclear. Also $\pi_2$ is nuclear as it equals $ J ( \: \cdot \: ) J \circ \pi_1 \circ \gamma \circ  J ( \: \cdot \: ) J$.
It therefore follows that the mapping  $\pi_1 \otimes \pi_2: (C(\partial W) \rtimes_r W) \otimes_{{\rm min}} J A_q J  \rightarrow \cB(\ell^2(W)) \slash \cK $ in diagram \eqref{Eqn=Diagram1} is nuclear and we may apply the Choi-Effros lifting theorem \cite{ChoiEffros} in order to obtain a ucp lift $\theta: (C(\partial W) \rtimes_r W) \otimes_{{\rm min}} J A_q J  \rightarrow \cB(\ell^2(W))$. Then $\theta \circ (\gamma \otimes \Id)$ together with \eqref{Eqn=ProductModCompacts} witness the result.
\end{proof}

\begin{cor}\label{Cor=NoCartanHyper}
Let $(W,S)$ be an irreducible hyperbolic Coxeter system with $\vert S \vert \geq 3$ and $q \in [\rho, \rho^{-1}]$. Then the Hecke von Neumann algebra $\cM_q$ has no Cartan subalgebra.
\end{cor}
\begin{proof}
This is a consequence of Theorem \ref{Thm=AOStrongSolidity} together with Theorems  \ref{Thm=NonInjective3}, \ref{Thm=CBAP} and \ref{Thm=APplus}.
\end{proof}

\begin{rmk}\label{Rmk=HyperbolicIsNecessary}
In case $W$ is not hyperbolic, it is not necessarily true that the group von Neumann algebra $\cM_1$ is strongly solid. The easiest case is when $\Gamma$ is $K_{2,3}$: the complete bipartite graph with 2+3 vertices. Then the graph product $W = \ast_{K_{2,3}} \mathbb{Z}_2 =(\mathbb{Z}_2 \ast \mathbb{Z}_2) \times (\mathbb{Z}_2 \ast \mathbb{Z}_2 \ast \mathbb{Z}_2)$  contains a copy of $\mathbb{Z} \times \mathbb{F}_2$. Then $\cM_1$ cannot be strongly solid as it contains the group von Neumann algebra of $\mathbb{Z} \times \mathbb{F}_2$. Note that $K_{2,3}$ is not an irreducible graph but  the same argument applies if one adds one point with no edges to $K_{2,3}$.
\end{rmk}

 \section{Absence of Cartan subalgebras} \label{Sect=Cartan}

As we saw in Remark \ref{Rmk=HyperbolicIsNecessary} the absence of Cartan subalgebras for general right-angled  Hecke von Neumann algebras  cannot be proved through strong solidity.  In this section we obtain absence of Cartan subalgebras for some additional Hecke von Neumann algebras through an analysis of amalgamated free products in conjunction with a  theorem by Vaes \cite[Theorem A]{VaesPrims} (see also \cite{IoanaENS} for related results).
We need some terminology first.

\begin{dfn}
Let $\cN, \mathcal{P} \subseteq \cM$ be finite von Neumann algebras. We say that $\cN$ is injective (or amenable) relative to $\mathcal{P}$ if there is a completely positive map $\Phi$ from the basic construction $\langle \cM, e_\mathcal{P} \rangle$ onto $\cN$ such that $\Phi\vert_\cM$ is the  conditional expectation of $\cM$ onto $\cN$. Here $e_{\mathcal{P}}$ is the Jones projection, i.e. the conditional expectation of $\cM$ to $\cP$ on the $L^2$-level.
\end{dfn}

The following Theorem \ref{Thm=Vaes} uses Popa's intertwining by bi-modules technique. For us it suffices that for finite (separable) von Neumann algebras $\cN, \cP \subseteq \cM$ we say that $\cN \prec_{\cM} \cP$ if there exists no sequence of unitaries $w_k$ in $\cN$ such that for all $x, y \in \cM$ we have $\Vert \mathcal{E}_{\cP} (x w_k y) \Vert_2 \rightarrow 0$. 
 The following theorem is a somewhat less general version of \cite[Theorem A]{VaesPrims}. 

\begin{thm}\label{Thm=Vaes}
Let $\cN_i, i = 1,2$ be finite von Neumann algebras with common von Neumann subalgebra $\cB$. Let $\cN = \cN_1 \ast_{\mathcal{B}} \cN_2$ be the (tracial) amalgamated free product. Let $\mathcal{A} \subseteq \mathcal{N}$ be a von Neumann subalgebra that is injective relative to one of the $\cN_i, i =1,2$. Then at least one of the following statements holds true: (1) $\mathcal{A} \prec_{\cN} \mathcal{B}$, (2) There exists $i$ such that $\Nor_{\mathcal{N}}(\mathcal{A}) \prec_\mathcal{N} \mathcal{N}_i$, (3) $\Nor_{\cN}(\mathcal{A})$ is injective relative to $\mathcal{B}$.
\end{thm}

Recall that for a graph $\Gamma$ and $r \in V\Gamma$ we have $\Link(r) = \{ s \in V \Gamma \mid (r,s) \in E\Gamma \}$ and $\Star(r) = \Link(r) \cup \{ r \}$. We include the following lemma to show that part of the condition in Theorem \ref{Thm=NoCartan} can always be achieved. 





 \begin{lem}\label{Lem=Disconnect}
Every irreducible   graph $\Gamma$ with $\vert V \Gamma \vert \geq 3$ contains a vertex $r \in V\Gamma$ such that $V\Gamma - \Star(r)$ contains at least two points.
\end{lem}
\begin{proof}
Pick some random point $r \in V\Gamma$. We cannot have $\Star(r) = V\Gamma$ because then $\Gamma$ would not be irreducible. So there is at least one point $w \in V\Gamma - \Star(r)$. If there is another point in $V\Gamma - \Star(r)$ then we are done, so we assume that $w$ is the only point in $V\Gamma - \Star(r)$. This implies that $\Link(r)$ is non-empty. $\Star(w)$ does not contain $r$ as $w \not \in \Star(r)$. Also there must be at least one point $u \in \Link(r)$ (which was non-empty!) that is not connected to $w$ because if this is not the case then every two elements in $\Link(r)$ and $\{r, w\}$ would be connected so that $\Gamma$ is not irreducible. In all we proved that $w$ has the property that $V \Gamma - \Star(w)$ contains at least two elements, namely $r$ and $u$.
\end{proof}

We recall the following definitions form \cite{CaspersFima}. 
\begin{dfn}
Let $\Gamma$ be a graph and let $w = w_1 \ldots w_n$ be a word with letters in $V\Gamma$. Suppose that $w_i = w_j$. We say that the $i$-th and $j$-th letter of $w$ are separated if there is a  $k$ with  $i < k < j$ such that $w_k \not \in \Star(w_i)$. If every two (equal) letters in $w$ are separated then $w$ is called {\it reduced}. 
\end{dfn}

\begin{dfn}\label{Dfn=ReducedOp}
	Let $\Gamma$ be a graph and for $s \in V\Gamma$ let $\cM(s)$ be a von Neumann algebra with normal faithful tracial state $\tau_s$. Let $\cM(s)^\circ = \{ a \in \cM(s)\mid \tau_s(a) = 0\}$.  Let $a = a_1 \ldots a_n$ with $a_i \in \cM(s_i)^\circ$ be an operator in the graph product von Neumann algebra $\star_{s \in V\Gamma} \cM(s)$. Then $a$ is called {\it reduced} if the word $s_1 \ldots s_n$ is reduced. The word $s_1 \ldots s_n$ is then called the {\it type} of $a$. We also say that two operators $a_i$ and $a_j$ of the same type $s \in V\Gamma$ are separated if there exists $i < k < j$ such that the type of $a_k$ is not in $\Star(s)$.
\end{dfn}

 \begin{dfn}
 	An inclusion of tracial von Neumann algebras $\cB \subseteq \cN$ is called mixing if for every sequence $b_n$ in $\cB$ with $\Vert b_n \Vert \leq 1$ and $b_n \rightarrow 0$ weakly we have that   $\Vert \cE_{\cB}(x b_n y) \Vert_2 \rightarrow 0$ for all $x,y \in \cN \ominus \cB$.
 \end{dfn}

For the proof of the following theorem we need a  condition assuming the existence of a specific point $r \in S$.  The condition is chosen such that in Claim 2 of the proof of Theorem \ref{Thm=NoCartan} we get a mixing inclusion of von Neumann algebras. This gives examples of Hecke von Neumann algebras of non-hyperbolic Coxeter groups that do not possess Cartan subalgebras. Indeed examples can easily be constructed; for example if there exists a point $r \in S$ such that $\Link(r)$ is the graph of a non-hyperbolic Coxeter group and if there are few edges between $\Link(r)$ and $V\Gamma - \Star(r)$ (i.e. such that the condition below is satisfied).  Though we believe that the theorem should hold without this condition we were unable to find a complete proof. 

\begin{thm}\label{Thm=NoCartan}
	Let $(W,S)$ be an irreducible right-angled Coxeter group with $\vert S \vert \geq 3$.  Let $q \in [\rho, \rho^{-1}]$. Assume that there is an element $r \in S$ such that 
	\begin{itemize}
		\item $V\Gamma - \Star(r)$ contains at least two points;
		\item For every $s,t \in \Link(r)$ such that $(s,t) \not \in E\Gamma$ we have that
		\[
		\Link(s) \cap \Link(t) \cap (V\Gamma - \Star(r)) = \emptyset.
		\]
	\end{itemize}
	Then the Hecke-von Neumann algebra $\cM_q$ does not have a Cartan subalgebra.
\end{thm}
\begin{proof}

	Let $\Gamma = (V \Gamma, E\Gamma)$ be the graph of $(W,S)$. By Corollary \ref{Cor=GraphDec} we get a graph product decomposition $\cM_q = \oast_{s \in V\Gamma} \cM_q(s)$ with $\cM_q(s)$ the Hecke-von Neumann algebra associated with the Coxeter subsystem generated by just $s$ (so it is 2-dimensional by Section \ref{Sect=Universal}).
	Choose $r \in V\Gamma$  satisfying the conditions   of the statement of the theorem.  Put 
	\[
	\cN_1 = \oast_{s \in \Star(r)} \cM_q(s), \quad \cN_2 = \oast_{s \in V\Gamma - \{ r \} }  \cM_q(s), \quad {\rm and } \quad \cB = \oast_{s \in \Link(r)} \cM_q(s).
	\]
	Here $\Link(r), \Star(r)$ and $V\Gamma - \{ r \}$ are all viewed as full subgraphs of $\Gamma$, i.e. a subgraph for which two vertices share an edge iff  they share an edge in $\Gamma$. 
	Simply write $\cM$ for $\cM_q$.   By  \cite[Theorem 2.26]{CaspersFima} we get:
	\[
	\cM = \cN_1  \ast_{\cB} \cN_2.
	\]
	Now suppose that $\mathcal{A} \subseteq \cM$ is a Cartan subalgebra. We are going to derive a contradiction by showing that any of the three alternatives of Theorem \ref{Thm=Vaes} is absurd.
	
	\vspace{0.3cm}

\noindent {\bf  Claim 1:} We cannot have  $\Nor_{\cM}(\mathcal{A}) \prec_{\mathcal{M}} \mathcal{N}_i$ for either $i = 1,2$.

\vspace{0.3cm}

\noindent {\it Proof of the claim.} As $\mathcal{A}$ is assumed to be Cartan we need to prove that $\cM \not \prec_{\mathcal{M}} \mathcal{N}_i$. Let $t \in V\Gamma - \Star(r)$. Then the subalgebra of $\cM$ generated by $\cM_q(r)$ and $\cM_q(t)$ is the tracial free product $\cM_q(r) \ast \cM_q(t)$. Take unitaries $u \in \cM_q(r)$ and $v \in \cM_q(t)$ with trace 0. Put $w_k = (uv)^k$ which then is a unitary in $\cM_q(r) \ast \cM_q(t)$ with trace 0.

  We need to show that for all $x,y \in \cM$ we have $\Vert \mathcal{E}_{\cN_i}( x w_k y ) \Vert_2 \rightarrow 0$. Recall that $\cM_q(s)^\circ$ is the space of elements $z \in \cM_q(s)$ with trace 0. By a density argument we may and will assume that $x = x_{1} \ldots x_k$ and $y = y_1 \ldots y_l$ are reduced operators with $x_i, y_i \in \cM_q(s)^\circ$ for some $s$ (see Definition \ref{Dfn=ReducedOp} or \cite[Definition 2.10]{CaspersFima} for the notion of reduced operators).
  Take a decomposition $x = x'a$ where $x' = x_1 \ldots x_{m}$ and  $a = x_{m+1} \ldots x_k$, with $x_{m+1}, \ldots, x_k \in \cM_q(r)^\circ \cup \cM_q(t)^\circ$. We may assume that this decomposition is taken in such a way that the length of $a$ is maximal, in other words: the end of the expression $x'$ has (after possible commutations) no factors $x_i$ that come from $\cM_q(r)^\circ$ and $\cM_q(t)^\circ$. We take a similar decomposition for $y$.  We may write $y = b y'$ with $y' = y_{n+1} \ldots y_l$ and $b = y_1 \ldots y_n$ with $y_i, 1 \leq i \leq n$ elements of either $\cM_q(r)^\circ$ and $\cM_q(t)^\circ$. Again we may assume that this decomposition is maximal meaning that (after possible commutations) the expression $y'$ does not have factors at the start that come from either $\cM_q(r)^\circ$ or $\cM_q(t)^\circ$.

  Now write $x w_k y = x'(a w_k b) y'$. For $k$ big (in fact $k \geq m +n + 1$ suffices) we get that $a w_k b$ is not contained in $\cN_i$ for neither $i = 1, 2$. Indeed $a$ and $b$ can never cancel all the occurences of $u$ and $v$ in $w_k = (uv)^k$ so that $a w_k b \in \cM_q(r) \ast \cM_q(t) \ominus (\cM_q(r) \cup \cM_q(t))$. So $x w_k y =  x'(a w_k b) y' \not \in \cN_i$ for either $i = 1,2$.    Therefore $\Vert\mathcal{E}_{\cN_i}(x w_k  y) \Vert_2 \rightarrow 0$ as $k \rightarrow \infty$.

 \vspace{0.3cm}

\noindent {\bf Claim 2:} We do not have $\mathcal{A} \prec_{\cM} \cB$.

\vspace{0.3cm}

\noindent {\it Proof of the claim.}  
Firstly we check that the inclusion $\cB \subseteq \cN_2$ is mixing. Let $b_n$ be a sequence in $\cB$ with $\Vert b_n \Vert \leq 1$ such that $b_n \rightarrow 0$ weakly. Take $x,y \in \cN_2 \ominus \cB$. By linearity and density we may assume that both $x$ and $y$ are reduced operators. In particular write a reduced expression $x = x_1 \ldots x_n$ with $x_i \in \cM_q(s_i)^\circ$ for some $s_i \in V\Gamma - \{ r\}$ and $1 \leq i \leq n$. Since $x$ is not in $\cB$ let $x_{i_0}$ be such that $s_{i_0} \not \in \Link(r)$. Let $V\Lambda$ be the set of all vertices in $\Link(r)$ that share an edge with $s_{i_0}$. Let $\Lambda$ be the full subgraph of $\Gamma$ with edge set $V\Lambda$. Then $\Lambda$ must be complete (i.e. every two vertices share an edge) because otherwise this would contradict the assumptions on $r$. This means that $\widetilde{\cB} := \star_{s \in V\Lambda} \cM_q(s)  = \otimes_{s \in V\Gamma} \cM_q(s)$ is finite dimensional, as $\cM_q(s)$ is 2-dimensional, see Section \ref{Sect=Universal}. This in turn implies that 
$\Vert \cE_{\widetilde{\cB}}(b_n) \Vert_2 \rightarrow 0$ (indeed, $b_n$ is bounded and converges to 0 weakly, hence $\sigma$-weakly; so $\cE_{\widetilde{\cB}}(b_n) \rightarrow 0$ $\sigma$-weakly and hence in the $\Vert \: 
\cdot \: \Vert_2$-norm, by finite dimensionality).
Now we have
\[
\cE_{\cB}(x b_n y ) = 
\cE_{\cB}(x (b_n - \cE_{\widetilde{\cB}}(b_n)  ) y ) + 
\cE_{\cB}(x  \cE_{\widetilde{\cB}}(b_n)  y ),
\]  
where the second summand thus converges to 0 in the $\Vert \: \cdot \: \Vert_2$-norm as $n \rightarrow \infty$. Further $\cE_{\cB}(x (b_n - \cE_{\widetilde{\cB}}(b_n)  ) y ) = 0$ for every $n$ as the operator  $x_{s_{i_0}}$ is separated from  any other operator of type $s_{i_0}$. So this shows that   $\Vert \cE_{\cB}(x b_n y )  \Vert_2  = \Vert 
\cE_{\cB}(x  \cE_{\widetilde{\cB}}(b_n)    y ) \Vert_2 \rightarrow 0$. This concludes our claim that the inclusion $\cB \subseteq \cN_2$ is mixing.

If $\mathcal{A} \prec_{\cM} \cB$ then we certainly have $\mathcal{A} \prec_\cM \cN_2$. But then by \cite[Lemma 9.4]{IoanaENS} and the previous paragraph which shows that the inclusion $\cN_i \subseteq \cM$ is mixing,   we get that also $\Nor_{\cM}(\cA)  \prec_\cM \cN_2$. However this is impossible by Claim 1.

 \vspace{0.3cm}

\noindent {\bf Claim 3:} $\cM$ is not relatively injective with respect to $\mathcal{B}$.

\vspace{0.3cm}

\noindent {\it Proof of the claim.} Recall our choice of $r \in V\Gamma$ at the start of the proof. Let $t_1,t_2$ be two different points in $V\Gamma - \Star(r)$. Let $\Lambda$ be the full subgraph of $\Gamma$ with vertex set $\{r, t_1, t_2 \}$. Let $\cN = \star_{s \in V\Lambda} \cM_q(s)$.
Note that $\cN \cap \cB = \mathbb{C}$.
Suppose that $\cM$ were to be relatively injective with respect to $\mathcal{B}$. Then there exists a (possibly non-normal) conditional expectation $\Phi: \langle \cM, e_{\mathcal{B}} \rangle \rightarrow \cM$. We shall prove that this implies that $\cN$ is injective.  

Let $A$ be the set of all reduced words $\boldw$ with letters in $V\Gamma$ that do not end on letters in $\Link(r)$ and that do not start with letters in $\{ r, t_1, t_2\}$, meaning that  for each $s \in \Link(r)$ the word $\boldw s$ is reduced and for each $s \in  \{ r, t_1, t_2\}$ the word $s \boldw$ is   reduced. For each word $\boldw \in W$ let $X_{\boldw}$ be a maximal set of reduced operators in $\cM$ of type $\boldw$ that form an orthonormal system in $L^2(\cM)$. Let $x \in X_{\boldw}, x' \in X_{\boldw'}$ with $\boldw, \boldw' \in A$ and $x \not = x'$. The spaces spanned by $\cN x \cB$ and $\cN x' \cB$ are orthogonal in $L^2(\cM)$ and invariant subspaces for $\cN$. Moreover, the projection\footnote{
Indeed $p_x$ is a projection: clearly $p_x^\ast = p_x$. Further, by assumption  on $x =x_1 \ldots x_k$ we have for $n \in \cN$ that $nx$ is a reduced operator. Take  $b \in \cB$ of trace 0.  The word $n_i x b$ is then reduced. In order to determine the conditional expectation  $\cE_{\cB}$ of  $x^\ast n_i^\ast  n_j x b$ one needs to write  $x^\ast n_i^\ast  n_j x b$ as a sum of reduced operators and delete all terms that are not in $\cB$. But the only such terms are the ones where $n_i^\ast$ annihilates $n_j$ and where each $x_i^\ast$ annihilates $x_i$. That is, 
\[
\cE_{\cB}(x^\ast n_i^\ast  n_j x b) = \tau(n_i^\ast n_j) \tau(x^\ast x)  b = \delta_{i,j} b. 
\]
Similarly, in order to determine $\cE_{\cB}(x^\ast n_i^\ast  n_j x b)$ one writes  $x^\ast n_i^\ast  n_j x$ as a reduced expression and filters all operators that are in $\cB$. Using that $x$ does not end on letters in $\cB$, this can only happen if $n_i^\ast$ annihilates the letter $n_j$ and each $x_i^\ast$ annihilates $x_i$. That is, 
\[
\cE_{\cB}(x^\ast n_i^\ast  n_j x) = \tau(n_i^\ast n_j) \tau(x^\ast x)   = \delta_{i,j}.
\]  
So we conclude $\cE_{\cB}(x^\ast n_i^\ast  n_j x b) = \delta_{i,j} b$ for any $b \in \cB$. This gives  $e_{\mathcal{B}} x^\ast n_i^\ast  n_j x e_{\mathcal{B}} = \delta_{i,j} e_{\mathcal{B}}$.
 Then
\[
p_x^2 = \sum_{i,j \in I}  n_i x e_{\mathcal{B}} x^\ast n_i^\ast
 n_j x e_{\mathcal{B}} x^\ast n_j^\ast
 = \sum_{i \in I}  n_i x e_{\mathcal{B}}   x^\ast n_i^\ast = p_x. 
\]
The image of $p_x$ is clearly contained in $\overline{ {\rm span} \: \cN x \cB}^{\Vert \: \Vert_2}$. Finally for a vector $nxb, n \in \cN, b \in \cB$ we have
\[
\begin{split}
p_x (nxb) = & \sum_{i \in I}  n_i x e_{\mathcal{B}}   x^\ast n_i^\ast nxb 
 =    \sum_{i \in I}  n_i x  \tau( x^\ast x) \tau(n_i^\ast n)   b =  \sum_{i \in I}  n_i x    \tau(n_i^\ast n)   b = nxb. 
\end{split}
\]
} 
of $L^2(\cM)$ onto $\overline{ {\rm span} \: \cN x \cB}^{\Vert \: \Vert_2}$ is given by 
\[
p_x = \sum_{i \in I}  n_i x e_{\mathcal{B}} x^\ast n_i^\ast,
\]
where we have chosen $n_i, i \in I$ to be elements of $\cN$ that form an orthonormal basis of $L^2(\cN)$. In particular $p_x \in \langle \cM, e_{\cB} \rangle$. We have that the projections $p_x, x \in X_{\boldw}, \boldw \in A$ commute with $\cN$ and they  sum up to 1 as 
\[
L^2(\cM) =  \bigoplus_{\boldw \in A, x \in X_{\boldw}} \overline{ {\rm span} \: \cN x \cB}^{\Vert \: \Vert_2}.
\]
For $\boldw \in A, x \in X_{\boldw}$ set,
\[
p_x' =   x e_{\mathcal{B}} x^\ast.
\]
Similarly, $p_x'$ is the projection onto $\overline{ {\rm span} \: x \cB}^{\Vert \: \Vert_2}$ and $p_x' \leq p_x$. 
 We claim that the von Neumann algebra generated by $p_x \cN p_x$ and $p_x'$ is homogeneous of type I. In order to do so note that there is a unitary\footnote{
Indeed $U_x$ is unitary as 
\[
\Vert \sum_{i} n_i x b_i \Vert_2^2 =
\sum_{i, j}  \tau( b_j^\ast  x^\ast  n_j^\ast n_i x b_i  )
=  \sum_{i,j} \tau(n_j^\ast n_i) \tau( b_j^\ast b_i ) 
= \Vert \sum_i n_i \otimes b_i \Vert_2^2,
\] 
where the second equality uses that $n_i x b_i$ is reduced by definition of $x$ and that $\tau(x^\ast x) = 1$ as $x$ had norm 1 in $L^2(\cM)$. 
} map:
\[
U_x: \overline{ {\rm span} \:  \cN x \cB}^{\Vert \: \Vert_2}  \rightarrow  L^2(\cN) \otimes L^2(\cB): n x b \mapsto n \otimes b.
\]
We have $U_x n U_x^\ast = n \otimes \Id_{L^2(\cB)}$ and  $U_x p_x' U_x^\ast = p_\Omega \otimes \Id_{L^2(\cB)}$ where $p_\Omega$ is the projection onto $\Omega := 1_{\cN} \in L^2(\cN)$. So that the von Neumann algebra   $U_x  \langle p_x \cN p_x, p_x' \rangle U_x^\ast$ is isomorphic to $\cB(L^2(\cN)) \otimes \Id_{L^2(\cB)}$, which is homogeneous of type I.

Now consider $\Psi:  \langle \cM, e_{\mathcal{B}} \rangle  \rightarrow^{\Phi}  \cM \rightarrow^{\mathcal{E}_\cN} \cN$. This is a conditional expectation for the inclusion $\cN \rightarrow \langle \cM, e_{\mathcal{B}} \rangle$. Let $\mathcal{P}$ be the subalgebra of $\langle \cM, e_{\mathcal{B}} \rangle$ that is generated by all $p_x \cN p_x$ and  $p_x'$ with $x \in X_\boldw, \boldw \in A$. The previous paragraph shows that $\cP = \bigoplus_{x \in X_{\boldw} ,\boldw \in A } \langle p_x \cN p_x, p_x' \rangle$ is homogeneous of type I. Restricting $\Psi$ to $\cP$  gives a conditional expectation for the inclusion $\cN \rightarrow \cP$ (recall that $\cN$ is contained in $\cP$ as the projections $p_x$ sum up to 1). Hence $\cN$ is an expected subalgebra of a homogeneous type I algebra. As homogeneous type I algebras are expected subalgebras of a type I factor we conclude that $\cN$ is injective.

\vspace{0.3cm}

\noindent {\it Remainder of the proof.}  Now Theorem \ref{Thm=Vaes} implies that either (1) $\Nor_{\cM}(\cA) \prec_\cM \cN_i$ for either $i =1$ or $i =2$; (2) $\mathcal{A} \prec_\cM \cB$; (3) $\cM$ is injective relative to $\cB$. The three claims above rule out all of these possibilities showing that $\cM$ does not possess a Cartan subalgebra.

\end{proof}

\bigskip

\footnotesize

\noindent
{\sc Martijn Caspers \\
Utrecht University, Budapestlaan 6, 3584 CD Utrecht, The Netherlands
\em E-mail address: \tt m.p.t.caspers@uu.nl}

\end{document}